\documentclass[oneside,reqno]{amsart}
\usepackage{amsmath,amssymb,amsthm}
\usepackage{cite}
\usepackage{a4wide}
\usepackage{color}
\usepackage{txfonts}
\usepackage{enumerate}
\DeclareMathAlphabet{\mathpzc}{OT1}{pzc}{m}{it}

\newtheorem{theorem}{Theorem}[section]
\newtheorem{lem}[theorem]{Lemma}
\newtheorem{corollary}[theorem]{Corollary}

\numberwithin{equation}{section}
\theoremstyle{definition}
\newtheorem{defn}[theorem]{Definition}

\newtheorem{rem}[theorem]{Remark}

\usepackage{hyperref}
\allowdisplaybreaks
\begin{document}

\title[Grand variable Herz-Morrey type Besove spaces] {Grand variable Herz-Morrey type Besove spaces and Triebel-Lizorkin spaces }

\author[Sultan et.al]{ Mehvish Sultan$^{1}$ and  Babar Sultan$^{2^*}$ }

\address{$^{1}$ Department of Mathematics , Capital University Of Science and Technology, Islamabad, Pakistan.}
\email{mehvishsultanbaz@gmail.com}%

\address{$^{2^*}$ Department of Mathematics, Quaid-I-Azam University, Islamabad 45320, Pakistan.}
\email{babarsultan40@yahoo.com}%

\date{\today}
\dedicatory{}%
\begin{abstract}
In the article, the boundedness of vector-valued sublinear operators in grand variable Herz-Morrey spaces $M \dot{K}_{ \lambda, p(\cdot)}^{\eta (\cdot), q), \theta}\left(\mathbb{R}^{n}\right)$ are obtained. Then grand variable Herz-Morrey type Besov and Triebel-Lizorkin spaces are defined. We will also prove the equivalent quasi-norms by Peetre's maximal operators in these spaces.
\end{abstract}
\maketitle

\section{ Introduction}
In variable exponent spaces, the Hardy-Littlewood maximal operator's boundedness is crucial. For instance, it is well known that if the Hardy-Littlewood maximal operator is bounded in variable exponent Lebesgue space, then numerous conclusions from classical harmonic analysis and function theory also apply for the variable exponent case; see \cite{ref4,ref5,ref6,ref26}. Moreover, a variety of variable exponent spaces are presented, including: Bessel potential spaces, Besov and Trieble-Lizorkin spaces, Hardy spaces, Herz spaces, grand variable Herz spaces, grand variable weighted Herz spaces, Herz-Morrey spaces, grand variable Herz-Morrey spaces, Morrey spaces, Morrey type Besov and Trieble-Lizorkin spaces, Trieble-Lizorkin-Morrey spaces, Trieble-Lizorkin-Morrey spaces, and so forth; see \cite{ref1, ref2,ref3,ref7,ref10,ref11,ref13,ref16,ref18,ref19,ref19,ref27,ref29,ref18,ref19,ref32,ref33,ref41,ref42,someaims,communication,forum,fractal,hardyaims,fractionalaims} and references therein.  As you can see from \cite{ref12,ref21,ref22,ref23,ref24}, numerous conclusions about the boundedness of sublinear operators in these spaces have been established.

A sublinear operator $T$ satisfy the size condition
$$
|T g(x)| \leqslant C \int_{\mathbb{R}^{n}}|x-y|^{-n}|g(y)| \mathrm{d} y
$$
for all $g \in L_{\text {loc }}^{1}\left(\mathbb{R}^{n}\right)$ with compact support and a.e. $x \notin \operatorname{supp} g$. Then, $T$ is bounded on the grand variable Herz-Morrey spaces and the homogeneous and non-homogeneous Herz spaces (see the monographs \cite{ref25,notes}). L. Tang and D. Yang then expand these conclusions for the weighted vector-valued situation in \cite{ref35}. Herz type Besov and Triebel-Lizorkin spaces with variable exponent $\dot{K}_{p(\cdot)}^{\eta , q}\left(\mathbb{R}^{n}\right)$ were introduced by C. Shi and the second author in \cite{ref33}. $K_{p(\cdot)}^{\eta , q}\left(\mathbb{R}^{n}\right) and B_{\kappa }^{s}$ $\dot{K}_{p(\cdot)}^{\eta , q}\left(\mathbb{R}^{n}\right) and B_{\kappa }^{s}$ $K_{p(\cdot)}^{\eta , q}\left(\mathbb{R}^{n}\right) F_{\kappa }^{s}$ and $F_{\kappa }^{s}$ and their corresponding quasi-norms. For the Besov and Triebel-Lizorkin spaces of constant exponent Herz type, we direct the reader to \cite{ref37,ref38,ref39,ref40}. M. Izuki \cite{ref14} used variable exponent $M \dot{K}_{q, p(\cdot)}^{\eta , \lambda}\left(\mathbb{R}^{n}\right)$ to get the vector-valued boundedness for certain sublinear operators that meet the size requirement on Herz-Morrey spaces. Sultan et al. \cite{axioms} presented the concept of grand  variable Herz-Morrey spaces; see \cite{desscience,ankara,marcinkiewiczaims} for additional findings on these spaces. The boundedness of vector-valued Hardy-Littlewood maximal operators in Herz spaces with variable exponents was established by the authors in \cite{ref8}. Peetre's maximal operators were used to characterize Herz type Besov and Triebel-Lizorkin spaces with variable exponents.

The current research examines the boundedness of vector-valued Hardy-Littlewood maximal operators on grand variable Herz-Morrey type Besov and Triebel-Lizorkin spaces, drawing inspiration from the aforementioned publications. The paper is organised as follows. We provide some conventions in the remainder of the section. Our primary findings, which are a generalization of related findings in \cite{ref14} and \cite{ref8,besov}, will be presented in Section 2. We provide evidence for our findings in Section 3.

\section {Main results}

The $n$-dimensional real Euclidean space is denoted as $\mathbb{R}^n$, as is customary. Let $\mathcal{A}$ be a measurable subset of $\mathbb{R}^n$.  Throughout this paper, $C, c$ are always positive constants which may vary from line to line. For a measurable set $S$, $|S|$ denote its Lebesgue measure and ${\bf{1}}_S$ denote its characteristic function. We write $u \leqslant v$, if $u\leq cv$, and if $u\leqslant v$ and $v\leqslant u$, then $u\sim v$. We assume a measurable function on $\mathbb{R}^n$ by the variable exponents. For every $l \in \mathbb{Z},$ we denote $B_t =B(0,2^t)=\{x \in \mathbb{R}^n : |x| < 2^t\}$. After deducting $B_{t-1}$ from $R_t$, ${\bf{1}}_{R_t}={\bf{1}} _ t$. By $\mathcal{P}(\mathbb{R}^n)$  we denote the subset of variable exponents with range $[1,\infty]$. Let ${q}_-:=\operatorname{ess} \inf \limits _{y\in \mathcal{A}}\;q(y)> 1$ and ${q}_+:=\operatorname{ess} \sup \limits _{y\in \mathcal{A}}\;q(y)< \infty$, then we have
	\begin{equation}\label{eq1}
	1\leq q_-(\mathcal{A})\leq q(y) \leq q_+(\mathcal{A})<\infty.
	\end{equation} 
 The Notation $B$ is the ball such that $B(z,r):= \{y \in \mathcal{A} : |z-y| <r\}$.
Now variable  Lebesgue space  $L^{q(\cdot)}(\mathcal{A})$ is given as
		$$L^{p(\cdot)}(\mathcal{A})=\left\{ {g}\; \textit{is measurable}: \int_{\mathcal{A}}\left({|{g}(z)| \over \lambda}\right)^{q(z)}dz < \infty\; \textit{where}\; \lambda \; \textit{is a constant} \right\}.$$ 
  Lebesgue space is equipped with the norm

  $$\|{g}\|_{L^{q(\cdot)}(\mathcal{A})}=\inf\left\{\lambda >0: \int_{\mathcal{A}}\left({|{g}(z)| \over \lambda}\right)^{q(z)}dz\leq 1\right\}.$$
 The Hardy-Littlewood maximal operator ${M}$ for ${g} \in L_{\mathrm{loc}}^1(\mathcal{A})$ is defined as 
$${M}{g}(z):= \sup \limits _{0<r} {1 \over r^n} \int \limits _{B(z,r)} |{g}(z)|d z \;\;\;\; (z \in \mathcal{A}).$$

 The notion of $\mathcal{B}(\mathcal{A})$ denote  the collection of $q(\cdot)\in \mathcal{P}({\mathcal{A}})$ such that  ${M}$ is bounded on $ L^{q(\cdot)}(\mathcal{A})$. Let $x,y \in \mathcal{A}$ with $|x-y|\leq \frac{1}{2}$ and $q: \mathcal{A} \mapsto (0,\infty)$.  Now  $\mathrm{log}$-H\"older continuity condition (or Dini-Lipschitz condition) is given as

\begin{equation}\label{r5}
\frac{C}{-\ln|x-y|} \geq |q(x)-q(y)|,
\end{equation}
where $C$ is called $\mathrm{log}$-H\"older continuity constant.

  Let log-Dini-Lipschitz constant (or decay constant) $C_{\infty}>0$,  $q(\cdot)$ satisfiy  the decay condition if   $\lim\limits_{|x|\to \infty} q(x)=q_{\infty}:=q(\infty)$  such that

\begin{equation}\label{r6}
\frac{C_{\infty}}{\ln(e+|h|)}\geq |q(h)-q_\infty|.
\end{equation}
Let $C_0>0$,  $q(\cdot)$ satisfy the   $\log$ H\"older continuity condition at $0$  for  $|h|\leq\frac{1}{2}$, such that
\begin{equation}\label{r7}
\frac{C_0}{\ln|h|}\geq |q(h)-q(0)|.
\end{equation}

	 $\mathcal{P}^{\log} = \mathcal{P}^{\log}(\mathcal{A})$  consists of all functions $q(\cdot) \in \mathcal{P}(\mathcal{A})$ satisfying \eqref{eq1}  and \eqref{r5}.
	
	  $\mathcal{P}_\infty (\mathcal{A})$ (resp.  $\mathcal{P}_{0,\infty} (\mathcal{A})$) is the subset of $\mathcal{P}(\mathcal{A})$  consisting of functions which  satisfy  condition \eqref{r6} (resp.  both conditions \eqref{r6} and \eqref{r7}). The set of positive integers is denoted by $\mathbb{Z}_{+}$. For each $t_0 \in \mathbb{Z}, {\bf{1}}_{A_{t_0}}={\bf{1}}_{t_0}$; for $t_0 \in \mathbb{Z}_{+}$, ${\bf{1}}_{B_{0}}=\tilde{{\bf{1}}}_{0}$; and so on. ${\bf{1}}_{A_{t_0}}$ denote the characteristic function of $A_{t_0}$. By  $a \leqslant b$ we denote $a\leq C b$.

  Moreover, we define $\mathcal{P}^{0}\left(\mathbb{R}^{n}\right)$ to be the set of measurable functions $p$ on $\mathbb{R}^{n}$ with the range in $(0, \infty)$ such that $p_- >0$ and $p_+<\infty$. Given $p(\cdot) \in\mathcal{P}^{0}\left(\mathbb{R}^{n}\right)$, one can define the space $L^{p(\cdot)}\left(\mathbb{R}^{n}\right)$. This is equivalent to defining it to be the set of all functions $g$ such that
$$
\left|g\right|^{p_{0}} \in L^{p(\cdot) / p_{0}}\left(\mathbb{R}^{n}\right)
$$
where
$$
0<p_{0}<p_{-}, \quad \frac{p(\cdot)}{p_{0}} \in \mathcal{P}\left(\mathbb{R}^{n}\right).
$$
Then the quasi-norm is given as 
$$
\|g\|_{L^{p(\cdot)}\left(\mathbb{R}^{n}\right)}=\left\||g|^{p_{0}}\right\|_{L^{p(\cdot) / p_{0}}\left(\mathbb{R}^{n}\right)}^{1/ p_{0}}.
$$

\begin{lem}\cite{ref20}

Let $f$ belong to the Lebesgue space $L^{p(\cdot)}(\mathbb{R}^{n})$ and $g$ belong to $L^{p'(\cdot)}(\mathbb{R}^{n})$, where $p(\cdot)\in \mathcal{P}(\mathbb{R}^{n})$. Then, the product function $fg$ is integrable over $\mathbb{R}^{n}$, and the following inequality holds:

$$\int_{\mathbb{R}^{n}} \bigl\vert f(x)g(x) \bigr\vert \,dx \leq r_{p} \Vert f \Vert _{{p(\cdot )}} \Vert g \Vert _{{p'(\cdot)}},$$
where $r_{p}$ is defined as $r_{p}=1+\frac{1}{p^{-}} -\frac{1}{p^{+}}$.
\end{lem}

\begin{lem}\cite{uribe1} Let $\tilde{q}(\cdot)$ be a variable exponent defined as $\frac{1}{p(z)}-\frac{1}{q}=\frac{1}{\tilde{q}(z)} $ where $\left(z \in \mathbb{R}^{n}\right)$. Then for all measurable functions $f$ and $g$ we have

$$
\|f g\|_{{p(\cdot)}} \leq C\left\| g\right\|_{{\bar{q}(\cdot)}}\| g\|_{{q}}.
$$
\end{lem}

\begin{lem}\cite{izukicom}

Let $p(\cdot)$ be a function within the class $\mathcal{B}(\mathbb{R}^{n})$. For any ball $B$ in $\mathbb{R}^{n}$, we get

$$\frac{1}{ \vert B \vert } \Vert {\bf{1}}_{B} \Vert _{{p'(\cdot)}}\Vert {\bf{1}}_{B} \Vert _{{p(\cdot)}} \leqslant C,$$
where $C>0$ .
\end{lem}

\begin{lem}\cite{izukicom}

Assuming that $p(\cdot)$ is a function in the class $\mathcal{B}(\mathbb{R}^{n}),$ there exists a positive constant $C$ such that, for every ball $B$ in $\mathbb{R}^{n}$ and every measurable subset $S$ within $B,$ the following inequalities hold:

$$\frac{ \Vert {\bf{1}}_{B} \Vert _{{p(\cdot)}}}{ \Vert {\bf{1}}_{S} \Vert _{{p(\cdot)}}}\leqslant  \frac{ \vert B \vert }{ \vert S \vert },\;\;\; \frac{ \Vert {\bf{1}}_{S} \Vert _{{p(\cdot)}}}{ \Vert {\bf{1}}_{B} \Vert _{{p(\cdot)}}}\leqslant  \biggl( \frac{ \vert S \vert }{ \vert B \vert } \biggr)^{\omega _{1}}, \;\;\;\frac{ \Vert {\bf{1}}_{S} \Vert _{{p'(\cdot)}}}{ \Vert {\bf{1}}_{B} \Vert _{{p'(\cdot)}}}\leqslant  \biggl(\frac{ \vert S \vert }{ \vert B \vert } \biggr)^{\omega _{2}},$$

here  $0<\omega _{1}, \omega _{2}<1$.
\end{lem}

\begin{lem}\cite{ref4}
    
{\rm
 Let $\left\{g_{j}\right\}_{j=1}^{\infty}$ be the sequences of locally integrable functions, $1<r<\infty$ and   $p(\cdot) \in \mathcal{B}\left(\mathbb{R}^{n}\right)$. Then
}
$$
\left\|\left(\sum_{j=1}^{\infty}\left|M  g_{j}\right|^{r}\right)^{1 / r}\right\|_{p(\cdot)} \leqslant C\left\|\left(\sum_{j=1}^{\infty}\left|g_{j}\right|^{r}\right)^{1 / r}\right\|_{p(\cdot)}.
$$
\end{lem}

\begin{defn}

 Let $0<q<\infty, p(\cdot) \in \mathcal{P}\left(\mathbb{R}^{n}\right), 0 \leq \lambda<\infty$, $\theta>0$ and $\eta (\cdot): \mathbb{R}^{n} \rightarrow \mathbb{R}$ with $\eta  \in L^{\infty}\left(\mathbb{R}^{n}\right)$.\\
(i) The homogeneous grand variable Herz-Morrey space $M \dot{K}_{ \lambda, p(\cdot)}^{\eta (\cdot), q), \theta}\left(\mathbb{R}^{n}\right)$  is defined by

$$
M \dot{K}_{ \lambda, p(\cdot)}^{\eta (\cdot), q), \theta}\left(\mathbb{R}^{n}\right):=\left\{g\in  L_{\mathrm{loc}}^{p(\cdot)}\left(\mathbb{R}^{n} \backslash\{0\}\right):\|g\|_{M \dot{K}_{ \lambda, p(\cdot)}^{\eta (\cdot), q), \theta}\left(\mathbb{R}^{n}\right)}<\infty\right\}
$$

where

$$
\|g\|_{M \dot{K}_{ \lambda, p(\cdot)}^{\eta (\cdot), q), \theta}\left(\mathbb{R}^{n}\right)}:= \sup _{\epsilon>0}\sup _{L \in \mathbb{Z}} 2^{-L \lambda}\left( \epsilon^\theta \sum_{k=-\infty}^{L}\left\|2^{k \eta (\cdot)} g {\bf{1}}_{k}\right\|_{p(\cdot)}^{q(1+\epsilon)}\right)^{\frac{1}{q(1+\epsilon)}}.
$$

(ii) The non-homogeneous grand variable Herz-Morrey space $ MK_{\lambda, p(\cdot)}^{\eta (\cdot), q), \theta}\left(\mathbb{R}^{n}\right)$  is defined by

$$
 MK_{\lambda, p(\cdot)}^{\eta (\cdot), q), \theta}\left(\mathbb{R}^{n}\right):=\left\{g\in  L_{\mathrm{loc}}^{p(\cdot)}\left(\mathbb{R}^{n} \backslash\{0\}\right):\|g\|_{ MK_{\lambda, p(\cdot)}^{\eta (\cdot), q), \theta}\left(\mathbb{R}^{n}\right)}<\infty\right\}
$$

where

$$
\|g\|_{ MK_{\lambda, p(\cdot)}^{\eta (\cdot), q), \theta}\left(\mathbb{R}^{n}\right)}:=\sup_{\epsilon>0}\sup _{L \in \mathbb{N}_{0}} 2^{-L \lambda}\left(\epsilon^\theta\sum_{k=0}^{L}\left\|2^{k \eta (\cdot)} g {\bf{1}}_{k}\right\|_{p(\cdot)}^{q(1+\epsilon)}\right)^{\frac{1}{q(1+\epsilon)}}.
$$

\end{defn}

{\begin{lem}\label{lem2.1}\cite{ref32}
	Let $D>1$ and $q \in \mathcal{P}_{0,\infty}(\mathbb{R}^n)$. Then
	\begin{equation}\label{lemma 2.2 dr asad}
	\frac{1}{c_0} r^{\frac{n}{q(0)}} \leq \| {\bf{1}} _{B(0,Dr)\setminus B(0,r)} \| _{L^{q(\cdot)}} \leq c_0 r^{\frac{n}{q(0)}}, \; \textit{for}  \; 0<r \leq 1
	\end{equation}
	and 
	\begin{equation}\label{lemma 2.2 dr asad 2}
	\frac{1}{c_\infty} r^{\frac{n}{q_\infty}} \leq \| {\bf{1}} _{B(0,Dr)\setminus B(0,r)} \| _{L^{q(\cdot)}} \leq c_\infty r^{\frac{n}{q_\infty}}, \;\textit{for}  \; r \geq 1,
	\end{equation}
	respectively, where $c_0 \geq 1$ and $c_\infty \geq 1$ depend on $D$ but  not depend on $r$.
\end{lem}

\begin{theorem}

If $1<r<\infty$ and  $\eta (\cdot) \in L^{\infty}\left(\mathbb{R}^{n}\right) \cap$ $\mathcal{P}_{0}^{\log }\left(\mathbb{R}^{n}\right) \cap \mathcal{P}_{\infty}^{\log }\left(\mathbb{R}^{n}\right)$ with $\eta (0), \eta _{\infty} \in\left(-n \omega _{1}, n \omega _{2}\right)$, where $\omega _{1}, \omega _{2} \in(0,1)$ are constants appearing in \eqref{eq1}. Let  $p(\cdot) \in \mathcal{B}\left(\mathbb{R}^{n}\right),$ $0<q<\infty$, and $$0 \leqslant \lambda<\min \left\{\left(n \omega_{1}+\eta (0)\right) / 2,\left(n \omega_{1}+\eta _{\infty}\right) / 2\right\}.$$ Suppose that $T$ is a sublinear operator satisfying vector-valued inequality on $L^{p(\cdot)}\left(\mathbb{R}^{n}\right)$

\begin{equation}\label{eq2.2}
\left\|\left(\sum_{j=1}^{\infty}\left|T g_{j}\right|^{r}\right)^{\frac{1}{r}}\right\|_{p(\cdot)} \leqslant C\left\|\left(\sum_{j=1}^{\infty}\left|g_{j}\right|^{r}\right)^{\frac{1}{r}}\right\|_{p(\cdot)}
\end{equation}

for all sequences $\left\{g_{j}\right\}_{j=1}^{\infty}$ of locally integrable functions on $\mathbb{R}^{n}$. Then we have the vector-valued inequality on ${ M\dot{K}_{\lambda, p(\cdot)}^{\eta (\cdot), q), \theta}\left(\mathbb{R}^{n}\right)}$

$$
\left\|\left(\sum_{k=1}^{\infty}\left|T g_{k}\right|^{r}\right)^{\frac{1}{r}}\right\|_{{ M\dot{K}_{\lambda, p(\cdot)}^{\eta (\cdot), q), \theta}\left(\mathbb{R}^{n}\right)}} \leqslant C\left\|\left(\sum_{k=1}^{\infty}\left|g_{k}\right|^{r}\right)^{\frac{1}{r}}\right\|_{{ M\dot{K}_{\lambda, p(\cdot)}^{\eta (\cdot), q), \theta}\left(\mathbb{R}^{n}\right)}}.
$$

\end{theorem}

\begin{rem}
    
 Here and below, we only declare our main results in the homogeneous grand variable Herz-Morrey space because the proof for the non-homogeneous case can be treated by the similar way and is much more easier.
\end{rem}

\begin{lem}  (see \cite{ref4})
   Let $p(\cdot), r $, and  $\left\{g_{j}\right\}_{j=1}^{\infty}$ are given in Theorem $2.8$, then 

$$
\left\|\left(\sum_{j=1}^{\infty}\left|\mathcal{M} g_{j}\right|^{r}\right)^{\frac{1}{r}}\right\|_{L^{p(\cdot)}\left(\mathbb{R}^{n}\right)} \leqslant C\left\|\left(\sum_{j=1}^{\infty}\left|g_{j}\right|^{r}\right)^{\frac{1}{r}}\right\|_{L^{p(\cdot)}\left(\mathbb{R}^{n}\right)}.
$$
\end{lem}

From Theorem $2.5$ and Lemma $2.7$, we obtain the following result for the Hardy-Littlewood maximal operator.

\begin{corollary}

 Let $\eta , r, q,p $, are given in Theorem $2.8$  and $$0 \leqslant \lambda<\min \left\{\left(n \omega_{1}+\eta (0)\right) / 2,\left(n \omega_{1}+\eta _{\infty}\right) / 2\right\},$$ then

$$
\left\|\left(\sum_{k=1}^{\infty}\left|\mathcal{M} g_{k}\right|^{r}\right)^{\frac{1}{r}}\right\|_{{ M\dot{K}_{\lambda, p(\cdot)}^{\eta (\cdot), q), \theta}\left(\mathbb{R}^{n}\right)}} \leqslant C\left\|\left(\sum_{k=1}^{\infty}\left|g_{k}\right|^{r}\right)^{\frac{1}{r}}\right\|_{{ M\dot{K}_{\lambda, p(\cdot)}^{\eta (\cdot), q), \theta}\left(\mathbb{R}^{n}\right)}}.
$$

\end{corollary}

Let $\mathcal{S}\left(\mathbb{R}^n\right)$ denote the Schwartz functions and  $\mathcal{S}^{\prime}\left(\mathbb{R}^n\right)$ the set of all  tempered distributions. We define the Fourier transform of a function $g\in \mathcal{S}\left(\mathbb{R}^n\right)$ by 
$$\widehat{\varphi}(g)(y)=2\pi^{-n/2}\int\limits_{\mathbb{R}^n}e^{-i x\cdot y}g(x) dx, \; y \in \mathbb{R}^n,$$
while $\varphi^{\vee}$ is the inverse Fourier transform. Let $\varphi_{0} \in \mathcal{S}\left(\mathbb{R}^n\right)$ with $\varphi_{0}(y) \geqslant 0$ then we have
$$
\varphi_{0}(y)= \begin{cases}1, & |y| \leqslant 1 \\ 0, & |y| \geqslant 2.\end{cases}
$$
Let
$$
\varphi(y)=\varphi_{0}(y)-\varphi_{0}(2 y)
$$
and define
$$
\varphi_{\ell}(y)=\varphi\left(2^{-\ell} y\right), \quad  \ell \in \mathbb{N}.
$$
Then $\left\{\varphi_{\ell}\right\}_{\ell \in \mathbb{N}_{0}}$ be the resolution of unity, such that
$$
\sum_{\ell=0}^{\infty} \varphi_{\ell}(x)=1, \quad \ x \in \mathbb{R}^{n}.
$$

\begin{defn}

 Let $\left\{\varphi_{j}\right\}_{j \in \mathbb{N}_{0}}$ be a resolution of unity as above, $s \in \mathbb{R}, 0<\kappa , q \leq$ $\infty, p(\cdot) \in \mathcal{P}\left(\mathbb{R}^{n}\right)$ and $\eta (\cdot): \mathbb{R}^{n} \rightarrow \mathbb{R}$ with $\eta (\cdot) \in L^{\infty}\left(\mathbb{R}^{n}\right)$.

(i) Then the grand variable Herz-Morrey type Besov space is defined by

$$
{ M\dot{K}_{\lambda, p(\cdot)}^{\eta (\cdot), q), \theta}} B_{\kappa }^{s}\left(\mathbb{R}^{n}\right):=\left\{g\in  \mathcal{S}^{\prime}\left(\mathbb{R}^{n}\right):\|g\|_{\ell_{\kappa }\left({ M\dot{K}_{\lambda, p(\cdot)}^{\eta (\cdot), q), \theta}\left(\mathbb{R}^{n}\right)}\right)}<\infty\right\},
$$

where

$$
\|g\|_{{ M\dot{K}_{\lambda, p(\cdot)}^{\eta (\cdot), q), \theta}}  B_{\kappa }^{s}}:=\left\|\left\{2^{s j} \varphi_{j}^{\vee} *g\right\}_{j=0}^{\infty}\right\|_{\ell_{\kappa }\left({ M\dot{K}_{\lambda, p(\cdot)}^{\eta (\cdot), q), \theta}} \right)}.
$$

(ii) For $p_{+}<\infty$, the grand variable Herz-Morrey type Triebel-Lizorkin space is defined by

$$ M\dot{K}_{\lambda, p(\cdot)}^{\eta (\cdot), q), \theta} F_{\kappa }^{s}\left(\mathbb{R}^{n}\right):=\left\{g\in  \mathcal{S}^{\prime}\left(\mathbb{R}^{n}\right):\|g\|_{{ M\dot{K}_{\lambda, p(\cdot)}^{\eta (\cdot), q), \theta}}\left(\ell_{\kappa }\right)}<\infty\right\},
$$

where

$$
\|g\|_{{ M\dot{K}_{\lambda, p(\cdot)}^{\eta (\cdot), q), \theta}} F_{\kappa }^{s}}:=\left\|\left\{2^{s j} \varphi_{j}^{\vee} *g\right\}_{j=0}^{\infty}\right\|_{{ M\dot{K}_{\lambda, p(\cdot)}^{\eta (\cdot), q), \theta}}\left(\ell_{\kappa }\right)}.
$$

Here we denote respectively by $\ell_{\kappa }\left({ M\dot{K}_{\lambda, p(\cdot)}^{\eta (\cdot), q), \theta}}\right)$ and ${ M\dot{K}_{\lambda, p(\cdot)}^{\eta (\cdot), q), \theta}}\left(\ell_{\kappa }\right)$ the spaces of all sequences $\left\{g_{j}\right\}$ of measurable functions on $\mathbb{R}^{n}$ with finite quasi-norms

$$
\left\|\left\{g_{j}\right\}_{j=0}^{\infty}\right\|_{\ell_{\kappa }\left({ M\dot{K}_{\lambda, p(\cdot)}^{\eta (\cdot), q), \theta}}\right)}:=\left(\sum_{j=0}^{\infty}\left\|g_{j}\right\|_{{ M\dot{K}_{\lambda, p(\cdot)}^{\eta (\cdot), q), \theta}}}^{\kappa }\right)^{\frac{1}{\kappa }},
$$

and

$$
\left\|\left\{g_{j}\right\}_{j=0}^{\infty}\right\|_{{ M\dot{K}_{\lambda, p(\cdot)}^{\eta (\cdot), q), \theta}}\left(\ell_{\kappa }\right)}:=\left\|\left(\sum_{j=0}^{\infty}\left|g_{j}\right|^{\kappa }\right)^{1 / \kappa }\right\|_{{ M\dot{K}_{\lambda, p(\cdot)}^{\eta (\cdot), q), \theta}}}.
$$
\end{defn}
Let $S \geqslant0$,  $\varepsilon>0$ and  $\varPsi_{0}, \varPsi \in \mathcal{S}\left(\mathbb{R}^n\right)$ such that

\begin{gather}
\left|\widehat{\varPsi}_{0}(\varrho)\right|>0 \quad \text { on }\{|\varrho|<2 \varepsilon\} \label{eq2}\\
|\widehat{\varPsi}(\varrho)|>0 \quad \text { on }\left\{\frac{\varepsilon}{2}<|\varrho|<2 \varepsilon\right\} \label{eq3}
\end{gather}
and
\begin{equation}\label{eq4}
D^{\tau} \widehat{\varPsi}(0)=0, \;\it{for \;all}\;  |\tau| \leqslant S.
\end{equation}

Here, \eqref{eq2} and \eqref{eq3} are Tauberian conditions, while \eqref{eq4} expresses vanishing
moment conditions on $\varPsi$. 

In \cite{ref28}, J. Peetre introduced the classical Peetre’s maximal operator: 

Let a tempered distribution $g \in \mathcal{S}^{\prime}\left(\mathbb{R}^n\right)$, $a>0$, and $\left\{\Xi _{\ell}\right\}_{\ell\in  \mathbb{Z}} \subset \mathcal{S}\left(\mathbb{R}^n\right)$. Then system of maximal functions are given as
$$
\left(\Xi _{\ell}^{*}\right)_{a} g(x):=\sup _{y \in \mathbb{R}^n} \frac{\left|\Xi _{\ell} * g(x+y)\right|}{\left(1+2^{k}|y|\right)^{a}}, \quad x \in \mathbb{R}^n, \ell\in  \mathbb{Z}.
$$
Since $\Xi _{\ell} * g(y)$ makes sense pointwise, everything is well defined. We will often use dilates
$$
\Xi ^\prime_{\ell}(x)=2^{k n} \Xi \left(2^{k} x\right)
$$
of a fixed function $\Xi  \in \mathcal{S}\left(\mathbb{R}^n\right)$, where $\Xi _{0}(x)$ might be given by a separate function. Also continuous dilates are needed. If
$$
\Xi _{t}:=t^{-n} \Xi \left(t^{-1} \cdot\right).
$$
Then 
$$
\varPsi_{t,a}^{*}g(x):=\sup _{y \in \mathbb{R}^n} \frac{\left|\Xi _{t} * g(x+y)\right|}{\left(1+\frac{|y|}{t}\right)^{a}} \quad x \in \mathbb{R}^n, t>0.
$$

\begin{theorem}
If $\kappa , q \in(0, \infty], \omega >0$, $ s \in \mathbb{R}$ with $s<S+1$, and $\eta , q,p $, are same as given in Lemma  $2.9$. Let $p(\cdot) / p_{0} \in \mathcal{B}\left(\mathbb{R}^{n}\right)$ with $p_{0}<\min \left(p_{-}, 1\right)$.  Let $\Theta _{0}, \Theta $ $\in $ $\mathcal{S}\left(\mathbb{R}^{n}\right)$ be given by \eqref{eq2} and \eqref{eq3}, respectively. Then

(i) For $a>n / p_{0}$, then the space ${ M\dot{K}_{\lambda, p(\cdot)}^{\eta (\cdot), q), \theta}}B_{\kappa }^{s}\left(\mathbb{R}^{n}\right)$ can be characterized by

$$
{ M\dot{K}_{\lambda, p(\cdot)}^{\eta (\cdot), q), \theta}}B_{\kappa }^{s}\left(\mathbb{R}^{n}\right)=\left\{g\in  \mathcal{S}^{\prime}\left(\mathbb{R}^{n}\right):\|g\|_{{ M\dot{K}_{\lambda, p(\cdot)}^{\eta (\cdot), q), \theta}}}^{(i)} B_{\kappa }^{s}<\infty\right\}, \quad i=1, \cdots, 4
$$

where

\begin{align*}
& \|g\|_{{ M\dot{K}_{\lambda, p(\cdot)}^{\eta (\cdot), q), \theta}}B_{\kappa }^{s}}^{(1)}:=\left\|\Phi_{0} *g\right\|_{M\dot{K}_{\lambda, p(\cdot)}^{\eta (\cdot), q), \theta}}+\left(\int_{0}^{1} t^{-s \kappa }\left\|\Phi_{t} *g\right\|_{M\dot{K}_{\lambda, p(\cdot)}^{\eta (\cdot), q), \theta}}^{\kappa } \frac{\mathrm{d} t}{t}\right)^{1 / \kappa } \\
& \|g\|_{M\dot{K}_{\lambda, p(\cdot)}^{\eta (\cdot), q), \theta} B_{\kappa }^{s}}^{(2)}:=\left\|\left(\Phi_{0}^{*} g\right)_a\right\|_{M \dot{K}_{q, p}^{\eta (\cdot), \lambda}}+\left(\int_{0}^{1} t^{-s \kappa }\left\|\left(\Phi_{t}^{*} g\right)_{a}\right\|_{M\dot{K}_{\lambda, p(\cdot)}^{\eta (\cdot), q), \theta}}^{\kappa } \frac{\mathrm{d} t}{t}\right)^{1 / \kappa } \\
& \|g\|_{{M\dot{K}_{\lambda, p(\cdot)}^{\eta (\cdot), q), \theta}}B_{\kappa }^{s}}^{(3)}:=\left(\sum_{k=0}^{\infty} 2^{s k \kappa }\left\|\left(\Phi_{k}^{*} g\right)_{a}\right\|_{M\dot{K}_{\lambda, p(\cdot)}^{\eta (\cdot), q), \theta}}^{\kappa }\right)^{1 / \kappa } \\
& \|g\|_{{M\dot{K}_{\lambda, p(\cdot)}^{\eta (\cdot), q), \theta}}B_{\kappa }^{s}}^{(4)}:=\left(\sum_{k=0}^{\infty} 2^{s k \kappa }\left\|\Phi_{k} *g\right\|_{{M\dot{K}_{\lambda, p(\cdot)}^{\eta (\cdot), q), \theta}}}^{\kappa }\right)^{1 / \kappa }.
\end{align*}

Then, $\left\{\|\cdot\|_{{M\dot{K}_{\lambda, p(\cdot)}^{\eta (\cdot), q), \theta}} B_{\kappa }^{s}}^{(i)}\right\}_{i=1}^{4}$ are equivalent.

(ii) If $p_{0}<\kappa $, then for $a>n / p_{0}$ the space ${M\dot{K}_{\lambda, p(\cdot)}^{\eta (\cdot), q), \theta}} F_{\kappa }^{s}\left(\mathbb{R}^{n}\right)$ can be characterized by

$$
{M\dot{K}_{\lambda, p(\cdot)}^{\eta (\cdot), q), \theta}}F_{\kappa }^{s}\left(\mathbb{R}^{n}\right)=\left\{g\in  \mathcal{S}^{\prime}\left(\mathbb{R}^{n}\right):\|g\|_{{M\dot{K}_{\lambda, p(\cdot)}^{\eta (\cdot), q), \theta}} F_{\kappa }^{s}}^{(i)}<\infty\right\}, i=1, \ldots, 5
$$

where

\begin{align}
\|g\|_{{M\dot{K}_{\lambda, p(\cdot)}^{\eta (\cdot), q), \theta}} F_{\kappa }^{s}}^{(1)}:= & \left\|\Phi_{0} *g\right\|_{{M\dot{K}_{\lambda, p(\cdot)}^{\eta (\cdot), q), \theta}}} \notag\\
& +\left\|\left(\int_{0}^{1} t^{-s \kappa }\left|\Phi_{t} *g\right|^{\kappa } \frac{\mathrm{d} t}{t}\right)^{1 / \kappa }\right\|_{{M\dot{K}_{\lambda, p(\cdot)}^{\eta (\cdot), q), \theta}}} \label{eq2.6} \\
\|g\|_{{M\dot{K}_{\lambda, p(\cdot)}^{\eta (\cdot), q), \theta}} F_{\kappa }^{s}}^{(2)}:= & \left\|\left(\Phi_{0}^{*} g\right)_{a}\right\|_{{M\dot{K}_{\lambda, p(\cdot)}^{\eta (\cdot), q), \theta}}} \notag\\
& +\left\|\left(\int_{0}^{1}\left[t^{-s}\left(\Phi_{t}^{*} g\right)_{a}\right]^{\kappa } \frac{\mathrm{d} t}{t}\right)^{1 / \kappa }\right\|_{{M\dot{K}_{\lambda, p(\cdot)}^{\eta (\cdot), q), \theta}}} \label{eq2.7}\\
\|g\|_{{M\dot{K}_{\lambda, p(\cdot)}^{\eta (\cdot), q), \theta}} F_{\kappa }^{s}}^{(3)}:= & \left\|\Phi_{0} *g\right\|_{{M\dot{K}_{\lambda, p(\cdot)}^{\eta (\cdot), q), \theta}}} \left\|\left(\int_{0}^{1} t^{-s \kappa }\right.\right. \notag \\
& \left.\left.\times \int_{|z|<t}\left|\left(\Phi_{t} *g\right)(\cdot+z)\right|^{\kappa } \mathrm{d} z \frac{\mathrm{d} t}{t^{n+1}}\right)^{1 / \kappa } \right\|_{{M\dot{K}_{\lambda, p(\cdot)}^{\eta (\cdot), q), \theta}}} \label{eq2.8}\\
\|g\|_{{M\dot{K}_{\lambda, p(\cdot)}^{\eta (\cdot), q), \theta}} F_{\kappa }^{s}}^{(4)}:= & \left\|\left(\sum_{k=0}^{\infty}\left[2^{k s \kappa }\left(\Phi_{k}^{*} g\right)_{a}\right]^{\kappa }\right)^{1 / \kappa }\right\|_{{M\dot{K}_{\lambda, p(\cdot)}^{\eta (\cdot), q), \theta}}} \label{eq2.9} \\
\|g\|_{{M\dot{K}_{\lambda, p(\cdot)}^{\eta (\cdot), q), \theta}}F_{\kappa }^{s}}^{(5)}:= & \left\|\left(\sum_{k=0}^{\infty} 2^{k s \kappa }\left|\Phi_{k} *g\right|^{\kappa }\right)^{1 / \kappa }\right\|_{{M\dot{K}_{\lambda, p(\cdot)}^{\eta (\cdot), q), \theta}}} \label{eq2.10}.
\end{align}

Then, $\left\{\|\cdot\|_{{M\dot{K}_{\lambda, p(\cdot)}^{\eta (\cdot), q), \theta}} F_{\kappa }^{s}}^{(i)}\right\}_{i=1}^{5}$ are equivalent.

\end{theorem}

\section{ Proofs of the main results}
 We need the following lemmas to prove our main results,

\begin{lem}
    
 Let $p, q, \lambda,  \eta $, are given in Theorem $2.8$ and $\theta>0$,  then

\begin{align*}
& \|g\|_{M\dot{K}_{\lambda, p(\cdot)}^{\eta (\cdot), q), \theta}} \approx \max \left\{\sup_{\epsilon>0}\sup _{L \leq 0, L \in \mathbb{Z}} 2^{-L \lambda}\left(\epsilon^\theta\sum_{k=-\infty}^{L} 2^{k \eta (0) q(1+\epsilon)}\left\|g {\bf{1}}_{k}\right\|_{p(\cdot)}^{q(1+\epsilon)}\right)^{\frac{1}{q(1+\epsilon)}}\right. \\
& \left.\quad \sup _{L>0, L \in \mathbb{Z}}\sup_{\epsilon>0}\left[2^{-L \lambda}\left(\epsilon^\theta\sum_{k=-\infty}^{-1} 2^{k \eta (0) q(1+\epsilon)}\left\|g {\bf{1}}_{k}\right\|_{p(\cdot)}^{q(1+\epsilon)}\right)^{\frac{1}{q(1+\epsilon)}}+2^{-L \lambda}\left(\epsilon^\theta\sum_{k=0}^{L} 2^{k \eta _{\infty} q(1+\epsilon)}\left\|g {\bf{1}}_{k}\right\|_{p(\cdot)}^{q(1+\epsilon)}\right)^{\frac{1}{q(1+\epsilon)}}\right]\right\}
\end{align*}

\end{lem}

Lemma $3.1$ is similar to Proposition $3.8$ in \cite{ref1}. Indeed, when $\eta (\cdot) \in L^{\infty}\left(\mathbb{R}^{n}\right) \cap$ $\mathcal{P}_{0}^{\log }\left(\mathbb{R}^{n}\right) \cap \mathcal{P}_{\infty}^{\log }\left(\mathbb{R}^{n}\right)$, there exist positive constants $C_{1}, C_{2}$ such that if $k \leq 0$ and $x \in D_{k}$ then $C_{1} 2^{k \eta (0)} \leq 2^{k \eta (x)} \leq C_{2} 2^{k \eta (0)}$; if $k>1$ and $x \in D_{k}$ then $C_{1} 2^{k \eta _{\infty}} \leq 2^{k \eta (x)} \leq C_{2} 2^{k \eta _{\infty}}$. Thus, we obtain Lemma $3.1$.

\begin{proof}

Now we will give the proof of Theorem $2.8$. Let $\left(\sum_{k=1}^{\infty}\left|g_{k}\right|^{r}\right)^{\frac{1}{r}} \in {M\dot{K}_{\lambda, p(\cdot)}^{\eta (\cdot), q), \theta}}\left(\mathbb{R}^{n}\right)$. Using the Lemma $3.1$ we get

\begin{align*}
& \left\|\left(\sum_{j=1}^{\infty}\left|T g_{j}\right|^{r}\right)^{\frac{1}{r}}\right\|_{{M\dot{K}_{\lambda, p(\cdot)}^{\eta (\cdot), q), \theta}}} \\
& \approx \max \left\{\sup_{\epsilon>0}\sup _{L \leq 0, L \in \mathbb{Z}} 2^{-L \lambda}\left(\epsilon^\theta \sum_{k=-\infty}^{L} 2^{k \eta (0) q(1+\epsilon)}\left\|\left(\sum_{j=1}^{\infty}\left|T g_{j}\right|^{r}\right)^{\frac{1}{r}} {\bf{1}}_{k}\right\|_{p(\cdot)}^{q(1+\epsilon)}\right)^{\frac{1}{q(1+\epsilon)}}\right. \\
& \sup_{\epsilon>0}\sup _{L>0, L \in \mathbb{Z}}\left[2^{-L \lambda}\left(\epsilon^\theta\sum_{k=-\infty}^{-1} 2^{k \eta (0) q(1+\epsilon)}\left\|\left(\sum_{j=1}^{\infty}\left|T g_{j}\right|^{r}\right)^{\frac{1}{r}} {\bf{1}}_{k}\right\|_{p(\cdot)}^{q(1+\epsilon)}\right)^{\frac{1}{q(1+\epsilon)}}\right. \\
& \left.\left.+2^{-L \lambda}\left(\epsilon^\theta\sum_{k=0}^{L} 2^{k \eta _{\infty} q(1+\epsilon)}\left\|\left(\sum_{j=1}^{\infty}\left|T g_{j}\right|^{r}\right)^{\frac{1}{r}} {\bf{1}}_{k}\right\|_{p(\cdot)}^{q(1+\epsilon)}\right)^{\frac{1}{q(1+\epsilon)}}\right]\right\} \\
& =: \max \left\{E_{T}, F_{T}\right\}.
\end{align*}

We also denote $F_{T}$ by $F_{T}:=\sup \limits _{\epsilon>0}\sup \limits _{L>0, L \in \mathbb{Z}}\left[G_{T}+H_{T}\right]$ with

\begin{align*}
G_{T} & :=2^{-L \lambda}\left\{\epsilon^\theta\sum_{k=-\infty}^{-1} 2^{k \eta (0) q(1+\epsilon)}\left\|\left(\sum_{j=1}^{\infty}\left|T g_{j}\right|^{r}\right)^{\frac{1}{r}} {\bf{1}}_{k}\right\|_{p(\cdot)}^{q(1+\epsilon)}\right\}^{\frac{1}{q(1+\epsilon)}} \\
H_{T} & :=2^{-L \lambda}\left\{\epsilon^\theta\sum_{k=0}^{L} 2^{k \eta _{\infty} q(1+\epsilon)}\left\|\left(\sum_{j=1}^{\infty}\left|T g_{j}\right|^{r}\right)^{\frac{1}{r}} {\bf{1}}_{k}\right\|_{p(\cdot)}^{q(1+\epsilon)}\right\}^{\frac{1}{q(1+\epsilon)}}.
\end{align*}

We need to prove that  $E_{T} \lesssim E_{f}, G_{T} \lesssim G_{f}$ and $H_{T} \lesssim H_{f}$ respectively, where

\begin{align*}
& E_{f}:=\sup_{\epsilon>0}\sup _{L \leqslant 0, L \in \mathbb{Z}} 2^{-L \lambda}\left\{\epsilon^\theta\sum_{k=-\infty}^{L} 2^{\eta (0) q(1+\epsilon) k}\left\|{\bf{1}}_{k}\left(\sum_{j=1}^{\infty}\left|g_{j}\right|^{r}\right)^{\frac{1}{r}}\right\|_{p(\cdot)}^{q(1+\epsilon)}\right\}^{\frac{1}{q(1+\epsilon)}} \\
& G_{f}:=2^{-L \lambda}\left\{\epsilon^\theta\sum_{k=-\infty}^{-1} 2^{k \eta (0) q(1+\epsilon)}\left\|\left(\sum_{j=1}^{\infty}\left|g_{j}\right|^{r}\right)^{\frac{1}{r}} {\bf{1}}_{k}\right\|_{p(\cdot)}^{q(1+\epsilon)}\right\}^{\frac{1}{q(1+\epsilon)}} \\
& H_{f}:=2^{-L \lambda}\left\{\epsilon^\theta\sum_{k=0}^{L} 2^{k \eta _{\infty} q(1+\epsilon)}\left\|\left(\sum_{j=1}^{\infty}\left|g_{j}\right|^{r}\right)^{\frac{1}{r}} {\bf{1}}_{k}\right\|_{p(\cdot)}^{q(1+\epsilon)}\right\}^{\frac{1}{q(1+\epsilon)}}.
\end{align*}

Hence we get  $E_{T} \lesssim E_{f}$ and $F_{T} \lesssim F_{g}$ where $F_{g}$ denote  $\sup \limits_{\epsilon>0}\sup \limits _{L>0, L \in \mathbb{Z}}\left[G_{f}+H_{f}\right]$. From above all and using Lemma $3.1$ again, we have

\begin{align*}
\left\|\left(\sum_{j=1}^{\infty}\left|T g_{j}\right|^{r}\right)^{\frac{1}{r}}\right\|_{M\dot{K}_{\lambda, p(\cdot)}^{\eta (\cdot), q), \theta}} & \approx \max \left\{E_{T}, F_{T}\right\} \\
& \lesssim \max \left\{E_{g}, F_{g}\right\} \\
& \approx\left\|\left(\sum_{j=1}^{\infty}\left|g_{j}\right|^{r}\right)^{\frac{1}{r}}\right\|_{{M\dot{K}_{\lambda, p(\cdot)}^{\eta (\cdot), q), \theta}}}.
\end{align*}

Estimate of $G_{T} \lesssim G_{f}$ is similar to $E_{\mathcal{M}} \lesssim E_{f}$ so omit the details.

 By using the size condition and   Minkowski's inequality, for $E_{T}$ we get

\begin{align*}
E_{T}= & \sup_{\epsilon>0}\sup _{L \leqslant 0, L \in \mathbb{Z}} 2^{-L \lambda}\left\{\epsilon^\theta \sum_{k=-\infty}^{L} 2^{\eta (0) q (1+\epsilon)k}\left\|{\bf{1}}_{k}\left(\sum_{j=1}^{\infty}\left|T g_{j}\right|^{r}\right)^{\frac{1}{r}}\right\|_{p(\cdot)}^{q(1+\epsilon)}\right\}^{\frac{1}{q(1+\epsilon)}} \\
= & \sup_{\epsilon>0}\sup _{L \leqslant 0, L \in \mathbb{Z}} 2^{-L \lambda}\left\{\epsilon^\theta\sum_{k=-\infty}^{L} 2^{\eta (0) q (1+\epsilon)k}\left\|{\bf{1}}_{k}\left(\sum_{j=1}^{\infty}\left|T \sum_{i=-\infty}^{\infty} g_{j}^{i}\right|^{r}\right)^{\frac{1}{r}}\right\|_{p(\cdot)}^{q(1+\epsilon)}\right\}^{\frac{1}{q(1+\epsilon)}} \\
\leqslant & \sup_{\epsilon>0} \sup _{L \leqslant 0, L \in \mathbb{Z}} 2^{-L \lambda}\left\{\epsilon^\theta\sum_{k=-\infty}^{L} 2^{\eta (0) q (1+\epsilon)k}\left\|{\bf{1}}_{k} \sum_{i=-\infty}^{\infty}\left(\sum_{j=1}^{\infty}\left|T g_{j}^{i}\right|^{r}\right)^{\frac{1}{r}}\right\|^{q(1+\epsilon)}\right\}_{p(\cdot)}^{\frac{1}{q(1+\epsilon)}} \\
& \lesssim \sup_{\epsilon>0}\sup _{L \leqslant 0, L \in \mathbb{Z}} 2^{-L \lambda}\left\{\epsilon^\theta\sum_{k=-\infty}^{L} 2^{\eta (0) q (1+\epsilon)k}\left\|{\bf{1}}_{k} \sum_{i=-\infty}^{k-2}\left(\sum_{j=1}^{\infty}\left|T g_{j}^{i}\right|^{r}\right)^{\frac{1}{r}}\right\|_{p(\cdot)}^{q(1+\epsilon)}\right\}^{\frac{1}{q(1+\epsilon)}} \\
& +\sup_{\epsilon>0}\sup _{L \leqslant 0, L \in \mathbb{Z}} 2^{-L \lambda}\left\{\epsilon^\theta\sum_{k=-\infty}^{L} 2^{\eta (0) q(1+\epsilon) k}\left\|{\bf{1}}_{k} \sum_{i=k-1}^{k+1}\left(\sum_{j=1}^{\infty}\left|T g_{j}^{i}\right|^{r}\right)^{\frac{1}{r}}\right\|_{p(\cdot)}^{q(1+\epsilon)}\right\}^{\frac{1}{q(1+\epsilon)}} \\
& +\sup_{\epsilon>0}\sup _{L \leqslant 0, L \in \mathbb{Z}} 2^{-L \lambda}\left\{\epsilon^\theta\sum_{k=-\infty}^{L} 2^{\eta (0) q (1+\epsilon)k}\left\|{\bf{1}}_{k} \sum_{i=k+2}^{\infty}\left(\sum_{j=1}^{\infty}\left|T g_{j}^{i}\right|^{r}\right)^{\frac{1}{r}}\right\|_{p(\cdot)}^{q(1+\epsilon)}\right\}^{\frac{1}{q(1+\epsilon)}} \\
=: & E_{T}^{1}+E_{T}^{2}+E_{T}^{3} .
\end{align*}

By the same way we consider $H_{T}$.

\begin{align*}
H_{T} \lesssim & 2^{-L \lambda}\left\{\epsilon^\theta\sum_{k=0}^{L} 2^{\eta _{\infty} q(1+\epsilon) k}\left\|{\bf{1}}_{k} \sum_{i=-\infty}^{k-2}\left(\sum_{j=1}^{\infty}\left|T g_{j}^{i}\right|^{r}\right)^{\frac{1}{r}}\right\|_{p(\cdot)}^{q}\right\} ^{\frac{1}{q(1+\epsilon)}}\\
& +2^{-L \lambda}\left\{\epsilon^\theta\sum_{k=0}^{L} 2^{\eta _{\infty} q(1+\epsilon) k}\left\|{\bf{1}}_{k} \sum_{i=k-1}^{k+1}\left(\sum_{j=1}^{\infty}\left|T g_{j}^{i}\right|^{r}\right)^{\frac{1}{r}}\right\|_{p(\cdot)}^{q(1+\epsilon)}\right\}^{\frac{1}{q(1+\epsilon)}} \\
& +2^{-L \lambda}\left\{\epsilon^\theta\sum_{k=0}^{L} 2^{\eta _{\infty} q(1+\epsilon) k}\left\|{\bf{1}}_{k} \sum_{i=k+2}^{\infty}\left(\sum_{j=1}^{\infty}\left|T g_{j}^{i}\right|^{r}\right)^{\frac{1}{r}}\right\|_{p(\cdot)}^{q(1+\epsilon)}\right\} ^{\frac{1}{q(1+\epsilon)}}\\
=: & H_{T}^{1}+H_{T}^{2}+H_{T}^{3} .
\end{align*}

Secondly, we will prove $E_{T}^{i}$ and $H_{T}^{i}, i=1,2,3$.

{\bf{Step 1.}} For  $E_{T}^{2}$, we have

\begin{align*}
& E_{T}^{2} \lesssim \sup_{\epsilon>0}\sup _{L \leqslant 0, L \in \mathbb{Z}} 2^{-L \lambda}\left\{\epsilon^\theta\sum_{k=-\infty}^{L} 2^{\eta (0) q(1+\epsilon) k} \sum_{i=k-1}^{k+1}\left\|\left(\sum_{j=1}^{\infty}\left|g_{j}^{i}\right|^{r}\right)^{\frac{1}{r}}\right\|_{p(\cdot)}^{q(1+\epsilon)}\right\} ^{\frac{1}{q(1+\epsilon)}}\\
& = \sup_{\epsilon>0}\sup _{L \leqslant 0, L \in \mathbb{Z}} 2^{-L \lambda}\left\{\epsilon^\theta\sum_{k=-\infty}^{L} 2^{\eta (0) q(1+\epsilon) k} \sum_{i=k-1}^{k+1}\left\|{\bf{1}}_{i}\left(\sum_{j=1}^{\infty}\left|g_{j}\right|^{r}\right)^{\frac{1}{r}}\right\|_{p(\cdot)}^{q(1+\epsilon)}\right\}^{\frac{1}{q(1+\epsilon)}}\\
& \lesssim \sup_{\epsilon>0} \sup _{L \leqslant 0, L \in \mathbb{Z}} 2^{-L \lambda}\left\{\epsilon^\theta\sum_{k=-\infty}^{L} 2^{\eta (0) q (1+\epsilon)k}\left\|{\bf{1}}_{k-1}\left(\sum_{j=1}^{\infty}\left|g_{j}\right|^{r}\right)^{\frac{1}{r}}\right\|_{p(\cdot)}^{q(1+\epsilon)}\right\}^{\frac{1}{q(1+\epsilon)}} \\
& + \sup_{\epsilon>0}\sup _{L \leqslant 0, L \in \mathbb{Z}} 2^{-L \lambda}\left\{\epsilon^\theta \sum_{k=-\infty}^{L} 2^{\eta (0) q (1+\epsilon)k}\left\|{\bf{1}}_{k}\left(\sum_{j=1}^{\infty}\left|g_{j}\right|^{r}\right)^{\frac{1}{r}}\right\|_{p(\cdot)}^{q(1+\epsilon)}\right\}^{\frac{1}{q(1+\epsilon)}} \\
& + \sup_{\epsilon>0}\sup _{L \leqslant 0, L \in \mathbb{Z}} 2^{-L \lambda}\left\{\epsilon^\theta\sum_{k=-\infty}^{L} 2^{\eta (0) q (1+\epsilon)k}\left\|{\bf{1}}_{k+1}\left(\sum_{j=1}^{\infty}\left|g_{j}\right|^{r}\right)^{\frac{1}{r}}\right\|_{p(\cdot)}^{q(1+\epsilon)}\right\}^{\frac{1}{q(1+\epsilon)}} \\
& \lesssim  \sup_{\epsilon>0} \sup _{L \leqslant 0, L \in \mathbb{Z}} 2^{-L \lambda}\left\{\epsilon^\theta\sum_{k=-\infty}^{L} 2^{\eta (0) q(1+\epsilon) k}\left\|{\bf{1}}_{k}\left(\sum_{j=1}^{\infty}\left|g_{j}\right|^{r}\right)^{\frac{1}{r}}\right\|_{p(\cdot)}^{q(1+\epsilon)}\right\} ^{\frac{1}{q(1+\epsilon)}}\\
& =: E_{f}.
\end{align*}

Then turn to $H_{T}^{2}$, similarly, we have

\begin{align*}
H_{T}^{2} & \lesssim 2^{-L \lambda}\left\{\epsilon^\theta\sum_{k=0}^{L} 2^{\eta _{\infty} q (1+\epsilon)k}\left\|{\bf{1}}_{k-1}\left(\sum_{j=1}^{\infty}\left|g_{j}\right|^{r}\right)^{\frac{1}{r}}\right\|_{p(\cdot)}^{q(1+\epsilon)}\right\} ^{\frac{1}{q(1+\epsilon)}}\\
& +2^{-L \lambda}\left\{\epsilon^\theta\sum_{k=0}^{L} 2^{\eta _{\infty} q(1+\epsilon) k}\left\|{\bf{1}}_{k}\left(\sum_{j=1}^{\infty}\left|g_{j}\right|^{r}\right)^{\frac{1}{r}}\right\|_{p(\cdot)}^{q(1+\epsilon)}\right\} ^{\frac{1}{q(1+\epsilon)}}\\
& +2^{-L \lambda}\left\{\epsilon^\theta\sum_{k=0}^{L} 2^{\eta _{\infty} q(1+\epsilon) k}\left\|{\bf{1}}_{k+1}\left(\sum_{j=1}^{\infty}\left|g_{j}\right|^{r}\right)^{\frac{1}{r}}\right\|_{p(\cdot)}^{q(1+\epsilon)}\right\} ^{\frac{1}{q(1+\epsilon)}}\\
& \lesssim H_{f}.
\end{align*}

{\bf{Step 2.}} 

Let $\forall i \leqslant k-2, x \in R_{k}, 1<r<\infty$, by the size condition and the generalized Minkowski's inequality, we obtain

\begin{align*}
\left(\sum_{j=1}^{\infty} T^{r}\left(g_{j}^{i}\right)(x)\right)^{\frac{1}{r}} & \lesssim\left[\sum_{j=1}^{\infty}\left(2^{-k n} \int_{\mathbb{R}^{n}}\left|g_{j}^{i}(y)\right| \mathrm{d} y\right)^{r}\right]^{\frac{1}{r}} \\
& \lesssim 2^{-k n} \int_{\mathbb{R}^{n}}\left(\sum_{j=1}^{\infty}\left|g_{j}^{i}\right|^{r}\right)^{\frac{1}{r}} \mathrm{~d} y.
\end{align*}

By Hölder's inequality we get

\begin{align}
E_{T}^{1} & \lesssim  \sup_{\epsilon>0}\sup _{L \leqslant 0, L \in \mathbb{Z}} 2^{-L \lambda}\left\{\epsilon^\theta \sum_{k=-\infty}^{L} 2^{\eta (0) q(1+\epsilon) k}\left\|{\bf{1}}_{k} \sum_{i=-\infty}^{k-2} 2^{-k n} \int_{\mathbb{R}^{n}}\left(\sum_{j=1}^{\infty}\left|g_{j}^{i}\right|^{r}\right)^{\frac{1}{r}} \mathrm{~d} y\right\|_{p(\cdot)}^{q(1+\epsilon)}\right\}^{\frac{1}{q(1+\epsilon)}} \notag\\
& =\sup_{\epsilon>0}\sup _{L \leqslant 0, L \in \mathbb{Z}} 2^{-L \lambda}\left\{\epsilon^\theta \sum_{k=-\infty}^{L} 2^{\eta (0) q (1+\epsilon)k}\left\|{\bf{1}}_{k}\right\|_{p(\cdot)}^{q(1+\epsilon)}\left(\sum_{i=-\infty}^{k-2} 2^{-k n} \int_{\mathbb{R}^{n}}\left(\sum_{j=1}^{\infty}\left|g_{j}^{i}\right|^{r}\right)^{\frac{1}{r}} \mathrm{~d} y\right)^{q(1+\epsilon)}\right\}^{\frac{1}{q(1+\epsilon)}} \notag\\
& \lesssim \sup_{\epsilon>0}\sup _{L \leqslant 0, L \in \mathbb{Z}} 2^{-L \lambda}\left\{\epsilon^\theta \sum_{k=-\infty}^{L} 2^{\eta (0) q (1+\epsilon)k}\left\|{\bf{1}}_{k}\right\|_{p(\cdot)}^{q(1+\epsilon)}\left(\sum_{i=-\infty}^{k-2} 2^{-k n}\left\|\left(\sum_{j=1}^{\infty}\left|g_{j}\right|^{r}\right)^{\frac{1}{r}} {\bf{1}}_{i}\right\|_{p(\cdot)}\left\|{\bf{1}}_{i}\right\|_{L^{p^{\prime}(\cdot)}}\right)^{q(1+\epsilon)}\right\}^{\frac{1}{q(1+\epsilon)}}\label{eq3.1}.
\end{align}

On the other hand, by using Lemma $2.4$, we have

\begin{equation}\label{eq3.2}
2^{-k n}\left\|{\bf{1}}_{k}\right\|_{p(\cdot)}\left\|{\bf{1}}_{i}\right\|_{L^{p^{\prime}(\cdot)}} \lesssim 2^{-k n}2^{\frac{kn}{p(0)}}2^{\frac{in}{p^\prime(0)}} \lesssim 2^{\frac{(i-k)n}{p^\prime(0)}}.
\end{equation}

We put \eqref{eq3.2} into \eqref{eq3.1} and get

\begin{align}
E_{T}^{1} & \lesssim \sup_{\epsilon>0}\sup _{L \leqslant 0, L \in \mathbb{Z}} 2^{-L \lambda}\left\{\epsilon^\theta\sum_{k=-\infty}^{L} 2^{\eta (0) k q(1+\epsilon)}\left(\sum_{i=-\infty}^{k-2}\left\|\left(\sum_{j=1}^{\infty}\left|g_{j}\right|^{r}\right)^{\frac{1}{r}} {\bf{1}}_{i}\right\|_{p(\cdot)} 2^{\frac{(i-k)n}{p^\prime(0)}}\right)^{q(1+\epsilon)}\right\}^{\frac{1}{q(1+\epsilon)}} \notag\\
& =\sup_{\epsilon>0}\sup _{L \leqslant 0, L \in \mathbb{Z}} 2^{-L \lambda}\left\{\epsilon^\theta\sum_{k=-\infty}^{L}\left(\sum_{i=-\infty}^{k-2} 2^{\eta (0) k}\left\|\left(\sum_{j=1}^{\infty}\left|g_{j}\right|^{r}\right)^{\frac{1}{r}} {\bf{1}}_{i}\right\|_{p(\cdot)} 2^{\frac{(i-k)n}{p^\prime(0)}}\right)^{q(1+\epsilon)}\right\}^{\frac{1}{q(1+\epsilon)}} \notag\\
& =\sup_{\epsilon>0}\sup _{L \leqslant 0, L \in \mathbb{Z}} 2^{-L \lambda}\left\{\epsilon^\theta\sum_{k=-\infty}^{L}\left(\sum_{i=-\infty}^{k-2} 2^{\eta (0) i}\left\|\left(\sum_{j=1}^{\infty}\left|g_{j}\right|^{r}\right)^{\frac{1}{r}} {\bf{1}}_{i}\right\|_{p(\cdot)} 2^{b(i-k)}\right)^{q(1+\epsilon)}\right\}^{\frac{1}{q(1+\epsilon)}}, \label{eq3.3}
\end{align}

here $b:=\frac{n}{p^\prime(0)}-\eta (0)>0$.

Let $1<q(1+\epsilon)<\infty$, then the Hölder's inequality yields

\begin{align*}
E_{T}^{1} & \lesssim \sup_{\epsilon>0} \sup _{L \leqslant 0, L \in \mathbb{Z}} 2^{-L \lambda}\left\{\epsilon^\theta \sum_{k=-\infty}^{L}\left(\sum_{i=-\infty}^{k-2} 2^{\eta (0) i q(1+\epsilon)}\left\|\left(\sum_{j=1}^{\infty}\left|g_{j}\right|^{r}\right)^{\frac{1}{r}} {\bf{1}}_{i}\right\|_{p(\cdot)}^{q(1+\epsilon)} 2^{\frac{b q(1+\epsilon)(i-k)}{2}}\right)\right. \\
& \left.\times\left(\sum_{i=-\infty}^{k-2} 2^{\frac{b (q(1+\epsilon))^{\prime}(i-k)}{2}}\right)^{\frac{q(1+\epsilon)}{(q(1+\epsilon))^{\prime}}}\right\}^{\frac{1}{q(1+\epsilon)}} \\
& \lesssim \sup_{\epsilon>0} \sup _{L \leqslant 0, L \in \mathbb{Z}} 2^{-L \lambda}\left\{\epsilon^\theta\sum_{k=-\infty}^{L} \sum_{i=-\infty}^{k-2} 2^{\eta (0) i q(1+\epsilon)}\left\|\left(\sum_{j=1}^{\infty}\left|g_{j}\right|^{r}\right)^{\frac{1}{r}} {\bf{1}}_{i}\right\|_{p(\cdot)}^{q(1+\epsilon)} 2^{\frac{b q(1+\epsilon)(i-k)}{2}}\right\}^{\frac{1}{q}} \\
& =\sup_{\epsilon>0}\sup _{L \leqslant 0, L \in \mathbb{Z}} 2^{-L \lambda}\left\{\epsilon^\theta\sum_{i=-\infty}^{L-2} 2^{\eta (0) i q(1+\epsilon)}\left\|\left(\sum_{j=1}^{\infty}\left|g_{j}\right|^{r}\right)^{\frac{1}{r}} {\bf{1}}_{i}\right\|_{p(\cdot)}^{q(1+\epsilon)} \sum_{k=i+2}^{L} 2^{\frac{b q(1+\epsilon)(i-k)}{2}}\right\}^{\frac{1}{q(1+\epsilon)}} \\
& \lesssim \sup_{\epsilon>0}\sup _{L \leqslant 0, L \in \mathbb{Z}} 2^{-L \lambda}\left\{\epsilon^\theta\sum_{i=-\infty}^{L-2} 2^{\eta (0) i q(1+\epsilon)}\left\|\left(\sum_{j=1}^{\infty}\left|g_{j}\right|^{r}\right)^{\frac{1}{r}} {\bf{1}}_{i}\right\|_{p(\cdot)}^{q(1+\epsilon)}\right\}^{\frac{1}{q(1+\epsilon)}} \\
& \lesssim E_{f} .
\end{align*}

If $0<q(1+\epsilon) \leqslant 1$, then we have

\begin{equation}\label{eq3.4}
\left(\sum_{i=1}^{\infty} a_{i}\right)^{q(1+\epsilon)} \leqslant \sum_{i=1}^{\infty} a_{i}^{q(1+\epsilon)}, \quad i \in \mathbb{N}, a_{i} \geqslant 0, 
\end{equation}

and obtain

\begin{align*}
E_{T}^{1} & \lesssim \sup_{\epsilon>0}\sup _{L \leqslant 0, L \in \mathbb{Z}} 2^{-L \lambda}\left\{\epsilon^\theta\sum_{k=-\infty}^{L} \sum_{i=-\infty}^{k-2} 2^{\eta (0) q(1+\epsilon) i}\left\|\left(\sum_{j=1}^{\infty}\left|g_{j}\right|^{r}\right)^{\frac{1}{r}} {\bf{1}}_{i}\right\|_{p(\cdot)}^{q(1+\epsilon)} 2^{b q(1+\epsilon)(i-k)}\right\}^{\frac{1}{q(1+\epsilon)}} \\
& = \sup_{\epsilon>0}\sup _{L \leqslant 0, L \in \mathbb{Z}} 2^{-L \lambda}\left\{\epsilon^\theta\sum_{i=-\infty}^{L-2} 2^{\eta (0) i q(1+\epsilon)}\left\|\left(\sum_{j=1}^{\infty}\left|g_{j}\right|^{r}\right)^{\frac{1}{r}} {\bf{1}}_{i}\right\|_{p(\cdot)}^{q(1+\epsilon)} \sum_{k=i+2}^{L} 2^{b q(1+\epsilon)(i-k)}\right\}^{\frac{1}{q(1+\epsilon)}} \\
& \lesssim  \sup_{\epsilon>0} \sup _{L \leqslant 0, L \in \mathbb{Z}} 2^{-L \lambda}\left\{\epsilon^\theta\sum_{i=-\infty}^{L-2} 2^{\eta (0) i q(1+\epsilon)}\left\|\left(\sum_{j=1}^{\infty}\left|g_{j}\right|^{r}\right)^{\frac{1}{r}} {\bf{1}}_{i}\right\|_{p(\cdot)}^{q(1+\epsilon)}\right\}^{\frac{1}{q(1+\epsilon)}} \\
& \lesssim E_{f} .
\end{align*}

Similarly, by using Lemmas $2.5$ and $3.2$, we have
\begin{equation*}
2^{-k n}\left\|{\bf{1}}_{k}\right\|_{p(\cdot)}\left\|{\bf{1}}_{i}\right\|_{L^{p^{\prime}(\cdot)}} \lesssim 2^{-k n}\left\|{\bf{1}}_{B_i}\right\|_{L^{p^\prime(\cdot)}}\left|B_k\right|\left\|{\bf{1}}_{B_k}\right\|^{-1}_{L^{p^{\prime}(\cdot)}} \lesssim 2^{n \omega_2(i-k)}.
\end{equation*}

$$
H_{T}^{1} \lesssim 2^{-L \lambda}\left\{\epsilon^\theta\sum_{k=0}^{\infty}\left(\sum_{i=-\infty}^{k-2} 2^{\eta _{\infty} i}\left\|\left(\sum_{j=1}^{\infty}\left|g_{j}\right|^{r}\right)^{\frac{1}{r}} {\bf{1}}_{i}\right\|_{p(\cdot)} 2^{b_{1}(i-k)}\right)^{q(1+\epsilon)}\right\}^{\frac{1}{q(1+\epsilon)}}
$$

here $b_{1}=n \omega_2-\eta _{\infty}>0$.

For $1<q(1+\epsilon)<\infty$, using Hölder's inequality, we have

$$
\begin{aligned}
H_{T}^{1} & \lesssim 2^{-L \lambda}\left\{\epsilon^\theta\sum_{k=0}^{L} \sum_{i=-\infty}^{k-2} 2^{\eta _{\infty} i q(1+\epsilon)}\left\|\left(\sum_{j=1}^{\infty}\left|g_{j}\right|^{r}\right)^{\frac{1}{r}} {\bf{1}}_{i}\right\|_{p(\cdot)}^{q(1+\epsilon)} 2^{\frac{b_{q}(i-k)}{2}}\left(\sum_{i=-\infty}^{k-2} 2^{\frac{b_{1} (q(1+\epsilon))^{\prime}(i-k)}{2}}\right)^{\frac{q(1+\epsilon)}{q(1+\epsilon)^{\prime}}}\right\}^{\frac{1}{q(1+\epsilon)}} \\
& \lesssim 2^{-L \lambda}\left\{\epsilon^\theta\sum_{k=0}^{L} \sum_{i=-\infty}^{k-2} 2^{\eta _{\infty} i q(1+\epsilon)}\left\|\left(\sum_{j=1}^{\infty}\left|g_{j}\right|^{r}\right)^{\frac{1}{r}} {\bf{1}}_{i}\right\|_{p(\cdot)}^{q(1+\epsilon)} 2^{\frac{b_{1} q(1+\epsilon)(i-k)}{2}}\right\}^{\frac{1}{q(1+\epsilon)}} \\
& \leqslant 2^{-L \lambda}\left\{\epsilon^\theta\sum_{k=0}^{L} \sum_{i=-2}^{k-2} 2^{\eta _{\infty} i q(1+\epsilon)}\left\|\left(\sum_{j=1}^{\infty}\left|g_{j}\right|^{r}\right)^{\frac{1}{r}} {\bf{1}}_{i}\right\|_{p(\cdot)}^{q} 2^{\frac{b_{1}q(1+\epsilon)(i-k)}{2}}\right\}^{\frac{1}{q(1+\epsilon)}} \\
& +2^{-L \lambda}\left\{\epsilon^\theta\sum_{k=0}^{L} \sum_{i=-\infty}^{-3} 2^{\eta _{\infty} i q(1+\epsilon)}\left\|\left(\sum_{j=1}^{\infty}\left|g_{j}\right|^{r}\right)^{\frac{1}{r}} {\bf{1}}_{i}\right\|_{p(\cdot)}^{q(1+\epsilon)} 2^{\frac{b_{1} q(1+\epsilon)(i-k)}{2}}\right\} ^{\frac{1}{q(1+\epsilon)}}\\
= & I_{1}+I_{2}.
\end{aligned}
$$

Now we consider $I_{1}$ and $I_{2}$ respectively. Due to $b_{1}>0$, we have

\begin{align*}
I_{1} & =2^{-L \lambda}\left\{\epsilon^\theta\sum_{i=-2}^{L-2} 2^{\eta _{\infty} i q(1+\epsilon)}\left\|\left(\sum_{j=1}^{\infty}\left|g_{j}\right|^{r}\right)^{\frac{1}{r}} {\bf{1}}_{i}\right\|_{p(\cdot)}^{q(1+\epsilon)} \sum_{k=i+2}^{L} 2^{\frac{b_{1} q(1+\epsilon)(i-k)}{2}}\right\}^{\frac{1}{q(1+\epsilon)}} \\
& \lesssim 2^{-L \lambda}\left\{\epsilon^\theta\sum_{i=-2}^{L-2} 2^{\eta _{\infty} i q(1+\epsilon)}\left\|\left(\sum_{j=1}^{\infty}\left|g_{j}\right|^{r}\right)^{\frac{1}{r}} {\bf{1}}_{i}\right\|_{p(\cdot)}^{q(1+\epsilon)}\right\}^{\frac{1}{q(1+\epsilon)}} \\
& \lesssim H_{f}.
\end{align*}

Because $L>0, b_{1}>0$ and $\lambda \geqslant 0$, we obtain

\begin{align*}
I_{2} & \lesssim 2^{-L \lambda}\left\{\epsilon^\theta \sum_{k=0}^{L} \sum_{i=-\infty}^{-3} 2^{\frac{b_{1} q(1+\epsilon)(i-k)}{2}}\left[\sum_{m=\infty}^{i} 2^{\eta _{\infty} m q(1+\epsilon)}\left\|\left(\sum_{j=1}^{\infty}\left|g_{j}\right|^{r}\right)^{\frac{1}{r}} {\bf{1}}_{m}\right\|_{p(\cdot)}^{q(1+\epsilon)}\right]\right\}^{\frac{1}{q(1+\epsilon)}} \\
& =2^{-L \lambda}\left\{\epsilon^\theta \sum_{k=0}^{L} \sum_{i=-\infty}^{-3} 2^{\frac{b_{1} q(1+\epsilon)(i-k)}{2}} \cdot 2^{i q(1+\epsilon) \lambda} H_{f}^{q(1+\epsilon)}\right\}^{\frac{1}{q(1+\epsilon)}} \\
& =2^{-L \lambda}\left\{\epsilon^\theta\left(\sum_{k=0}^{L} 2^{-k b_{1} q(1+\epsilon) / 2}\right)\left(\sum_{i=-\infty}^{-3} 2^{\left(b_{1} / 2+\lambda\right) q(1+\epsilon) i}\right) H_{f}^{q(1+\epsilon)}\right\}^{\frac{1}{q(1+\epsilon)}} \\
& \lesssim 2^{-L \lambda}\left\{\epsilon^\theta \left(\sum_{k=0}^{L} 2^{-k b_{1} q(1+\epsilon) / 2}\right)\left(\sum_{i=-\infty}^{L} 2^{\left(b_{1} / 2+\lambda\right) q(1+\epsilon) i}\right) H_{f}^{q(1+\epsilon)}\right\}^{\frac{1}{q(1+\epsilon)}} \\
& \lesssim 2^{-L \lambda} 2^{-L b_{1} / 2} 2^{\left(b_{1} / 2+\lambda\right) L} H_{f} \\
& =H_{f}.
\end{align*}

For $0<q(1+\epsilon) \leqslant 1$, using \eqref{eq3.4}

\begin{align*}
H_{T}^{1} & \lesssim 2^{-L \lambda}\left\{\epsilon^\theta \sum_{k=0}^{\infty} \sum_{i=-\infty}^{k-2} 2^{\eta _{\infty} i q(1+\epsilon)}\left\|\left(\sum_{j=1}^{\infty}\left|g_{j}\right|^{r}\right)^{\frac{1}{r}} {\bf{1}}_{i}\right\|_{p(\cdot)}^{q(1+\epsilon)} 2^{b_{1}(i-k)}\right\}^{\frac{1}{q(1+\epsilon)}} \\
& \lesssim 2^{-L \lambda}\left\{\epsilon^\theta \sum_{k=0}^{L} \sum_{i=-2}^{k-2} 2^{\eta _{\infty} i q(1+\epsilon)}\left\|\left(\sum_{j=1}^{\infty}\left|g_{j}\right|^{r}\right)^{\frac{1}{r}} {\bf{1}}_{i}\right\|_{p(\cdot)}^{q(1+\epsilon)} 2^{b_{1}(i-k)}\right\}^{\frac{1}{q(1+\epsilon)}} \\
& +2^{-L \lambda}\left\{\epsilon^\theta\sum_{k=0}^{L} \sum_{i=-\infty}^{-3} 2^{\eta _{\infty} i q(1+\epsilon)}\left\|\left(\sum_{j=1}^{\infty}\left|g_{j}\right|^{r}\right)^{\frac{1}{r}} {\bf{1}}_{i}\right\|_{p(\cdot)}^{q(1+\epsilon)} 2^{b_{1}(i-k)}\right\} ^{\frac{1}{q(1+\epsilon)}}\\
= & : J_{1}+J_{2}.
\end{align*}

Similar to Step $2$ $H_{T}^{2} \lesssim H_{f}$ is also true for $0<q(1+\epsilon) \leqslant 1$. 

{\bf{Step 3.}} Let  $\forall i \geq k+2, x \in R_{k}$, then we get

\begin{align*}
\left(\sum_{j=1}^{\infty} T^{r}\left(g_{j}^{i}\right)(x)\right)^{\frac{1}{r}} & \lesssim\left[\sum_{j=1}^{\infty}\left(2^{-i n} \int_{\mathbb{R}^{n}}\left|g_{j}^{i}(y)\right| \mathrm{d} y\right)^{r}\right]^{\frac{1}{r}} \\
& =2^{-i n}\left(\sum_{j=1}^{\infty}\left(\int_{\mathbb{R}^{n}}\left|g_{j}^{i}(y)\right| \mathrm{d} y\right)^{r}\right)^{\frac{1}{r}} \\
& \lesssim 2^{-i n} \int_{\mathbb{R}^{n}}\left(\sum_{j=1}^{\infty}\left|g_{j}^{i}\right|^{r}\right)^{\frac{1}{r}} \mathrm{~d} y.
\end{align*}

By Hölder's inequality we get

\begin{align}
E_{T}^{3} & \lesssim \sup_{\epsilon>0}\sup _{L \leqslant 0, L \in \mathbb{Z}} 2^{-L \lambda}\left\{\epsilon^\theta\sum_{k=-\infty}^{L} 2^{\eta (0) q(1+\epsilon) k}\left\|{\bf{1}}_{k} \sum_{i=k+2}^{\infty} 2^{-i n} \int_{\mathbb{R}^{n}}\left(\sum_{j=1}^{\infty}\left|g_{j}^{i}\right|^{r}\right)^{\frac{1}{r}} \mathrm{~d} y\right\|_{L^{p (\cdot)}}^{q(1+\epsilon)}\right\}^{\frac{1}{q(1+\epsilon)}}\notag \\
= & \sup_{\epsilon>0}\sup _{L \leqslant 0, L \in \mathbb{Z}} 2^{-L \lambda}\left\{\epsilon^\theta\sum_{k=-\infty}^{L} 2^{\eta (0) q(1+\epsilon) k}\left\|{\bf{1}}_{k}\right\|_{p(\cdot)}^{q(1+\epsilon)}\left(\sum_{i=k+2}^{\infty} 2^{-i n} \int_{\mathbb{R}^{n}}\left(\sum_{j=1}^{\infty}\left|g_{j}^{i}\right|^{r}\right)^{\frac{1}{r}} \mathrm{~d} y\right)^{q(1+\epsilon)}\right\}^{\frac{1}{q(1+\epsilon)}} \notag\\
& \lesssim\sup_{\epsilon>0} \sup _{L \leqslant 0, L \in \mathbb{Z}} 2^{-L \lambda}\left\{\epsilon^\theta\sum_{k=-\infty}^{L} 2^{\eta (0) q(1+\epsilon) k}\left\|{\bf{1}}_{k}\right\|_{p(\cdot)}^{q(1+\epsilon)}\left(\sum_{i=k+2}^{\infty} 2^{-i n}\left\|\left(\sum_{j=1}^{\infty}\left|g_{j}\right|^{r}\right)^{\frac{1}{r}} {\bf{1}}_{i}\right\|_{p(\cdot)}\left\|{\bf{1}}_{i}\right\|_{L^{{p}(\cdot)}}\right)^{q(1+\epsilon)}\right\}^{\frac{1}{q(1+\epsilon)}}.\label{eq3.5}
\end{align}

Using Lemmas $2.5$ and $3.2$ again, we obtain

\begin{align}
2^{-i n}\left\|{\bf{1}}_{k}\right\|_{p(\cdot)}\left\|{\bf{1}}_{i}\right\|_{L^{p^{\prime}(\cdot)}} & \leqslant 2^{-i n}\left\|{\bf{1}}_{B_{k}}\right\|_{p(\cdot)}\left\|{\bf{1}}_{B_{i}}\right\|_{L^{p^{\prime}(\cdot)}} \notag\\
& \lesssim 2^{-i n}\left\|{\bf{1}}_{B_{k}}\right\|_{p(\cdot)}\left|B_{i}\right|\left\|{\bf{1}}_{B_{i}}\right\|_{p(\cdot)}^{-1}\notag \\
& \lesssim 2^{n \omega_{1}(k-i)} \label{eq3.6}. 
\end{align}

We put \eqref{eq3.6} into \eqref{eq3.5} and get

\begin{align}
E_{T}^{3} & \lesssim \sup_{\epsilon>0}\sup _{L \leqslant 0, L \in \mathbb{Z}} 2^{-L \lambda}\left\{\epsilon^\theta\sum_{k=-\infty}^{L} 2^{\eta (0) k q(1+\epsilon)}\left(\sum_{i=k+2}^{\infty}\left\|\left(\sum_{j=1}^{\infty}\left|g_{j}\right|^{r}\right)^{\frac{1}{r}} {\bf{1}}_{i}\right\|_{p(\cdot)} 2^{n \omega_{1}(k-i)}\right)^{q(1+\epsilon)}\right\}^{\frac{1}{q(1+\epsilon)}} \notag\\
& =\sup_{\epsilon>0}\sup _{L \leqslant 0, L \in \mathbb{Z}} 2^{-L \lambda}\left\{\epsilon^\theta\sum_{k=-\infty}^{L}\left(\sum_{i=k+2}^{\infty} 2^{\eta (0) k}\left\|\left(\sum_{j=1}^{\infty}\left|g_{j}\right|^{r}\right)^{\frac{1}{r}} {\bf{1}}_{i}\right\|_{p(\cdot)} 2^{n \omega_{1}(k-i)}\right)^{q(1+\epsilon)}\right\}^{\frac{1}{q(1+\epsilon)}} \notag\\
& =\sup_{\epsilon>0}\sup _{L \leqslant 0, L \in \mathbb{Z}} 2^{-L \lambda}\left\{\epsilon^\theta\sum_{k=-\infty}^{L}\left(\sum_{i=k+2}^{\infty} 2^{\eta (0) i}\left\|\left(\sum_{j=1}^{\infty}\left|g_{j}\right|^{r}\right)^{\frac{1}{r}} {\bf{1}}_{i}\right\|_{p(\cdot)} 2^{d(k-i)}\right)^{q(1+\epsilon)}\right\}^{\frac{1}{q(1+\epsilon)}}. \label{eq3.7}
\end{align}

where $d:=n \omega_{1}+\eta (0)>0$.

Let $1<q(1+\epsilon)<\infty$, then we get 

$$
\begin{aligned}
E_{T}^{3} & \lesssim  \sup_{\epsilon>0}\sup _{L \leqslant 0, L \in \mathbb{Z}} 2^{-L \lambda}\left\{\epsilon^\theta\sum_{k=-\infty}^{L}\left(\sum_{i=k+2}^{\infty} 2^{\eta (0) i q(1+\epsilon)}\left\|\left(\sum_{j=1}^{\infty}\left|g_{j}\right|^{r}\right)^{\frac{1}{r}} {\bf{1}}_{i}\right\|_{p(\cdot)}^{q(1+\epsilon)} 2^{\frac{d q(1+\epsilon)(k-i)}{2}}\right)\left(\sum_{i=k+2}^{\infty} 2^{\frac{d (q(1+\epsilon))^{\prime}(k-i)}{2}}\right)^{\frac{q(1+\epsilon)}{(q(1+\epsilon))^{\prime}}}\right\}^{\frac{1}{q(1+\epsilon)}} \\
& \lesssim \sup_{\epsilon>0}\sup _{L \leqslant 0, L \in \mathbb{Z}} 2^{-L \lambda}\left\{\epsilon^\theta\sum_{k=-\infty}^{L} \sum_{i=k+2}^{\infty} 2^{\eta (0) i q(1+\epsilon)}\left\|\left(\sum_{j=1}^{\infty}\left|g_{j}\right|^{r}\right)^{\frac{1}{r}} {\bf{1}}_{i}\right\|_{p(\cdot)}^{q(1+\epsilon)} 2^{\frac{d q(1+\epsilon)(k-i)}{2}}\right\} ^{\frac{1}{q(1+\epsilon)}} \\
= & \sup_{\epsilon>0}\sup _{L \leqslant 0, L \in \mathbb{Z}} 2^{-L \lambda}\left\{\epsilon^\theta\sum_{k=-\infty}^{L} \sum_{i=k+2}^{L+2} 2^{\eta (0) i q(1+\epsilon)}\left\|\left(\sum_{j=1}^{\infty}\left|g_{j}\right|^{r}\right)^{\frac{1}{r}} {\bf{1}}_{i}\right\|_{p(\cdot)}^{q(1+\epsilon)} 2^{\frac{d q(1+\epsilon)(k-i)}{2}}\right\} ^{\frac{1}{q(1+\epsilon)}} \\
& +\sup_{\epsilon>0}\sup _{L \leqslant 0, L \in \mathbb{Z}} 2^{-L \lambda}\left\{\epsilon^\theta\sum_{k=-\infty}^{L} \sum_{i=L+3}^{\infty} 2^{\eta (0) i q(1+\epsilon)}\left\|\left(\sum_{j=1}^{\infty}\left|g_{j}\right|^{r}\right)^{\frac{1}{r}} {\bf{1}}_{i}\right\|_{p(\cdot)}^{q(1+\epsilon)} 2^{\frac{d q(1+\epsilon)(k-i)}{2}}\right\} ^{\frac{1}{q(1+\epsilon)}} \\
=: & I_{3}+I_{4} .
\end{aligned}
$$

Now we consider $I_{3}$ and $I_{4}$ respectively. For $d>0$, we get

\begin{align*}
I_{3} & =\sup_{\epsilon>0}\sup _{L \leqslant 0, L \in \mathbb{Z}} 2^{-L \lambda}\left\{\epsilon^\theta\sum_{i=-\infty}^{L+2} 2^{\eta (0) i q(1+\epsilon)}\left\|\left(\sum_{j=1}^{\infty}\left|g_{j}\right|^{r}\right)^{\frac{1}{r}} {\bf{1}}_{i}\right\|_{p(\cdot)}^{q(1+\epsilon)} \sum_{k=-\infty}^{i-2} 2^{\frac{d q(1+\epsilon)(k-i)}{2}}\right\}^{\frac{1}{q(1+\epsilon)}} \\
& \lesssim \sup_{\epsilon>0} \sup _{L \leqslant 0, L \in \mathbb{Z}} 2^{-L \lambda}\left\{\epsilon^\theta\sum_{i=-\infty}^{L+2} 2^{\eta (0) i q(1+\epsilon)}\left\|\left(\sum_{j=1}^{\infty}\left|g_{j}\right|^{r}\right)^{\frac{1}{r}} {\bf{1}}_{i}\right\|_{p(\cdot)}^{q(1+\epsilon)}\right\}^{\frac{1}{q(1+\epsilon)}} \\
& \lesssim E_{f}.
\end{align*}

If $d>0$ and $\lambda-d / 2<0$, then we get

\begin{align*}
I_{4} & \lesssim \sup_{\epsilon>0} \sup _{L \leqslant 0, L \in \mathbb{Z}} 2^{-L \lambda}\left\{\epsilon^\theta\sum_{k=-\infty}^{L} \sum_{i=L+3}^{\infty} 2^{\frac{d q(1+\epsilon)(k-i)}{2}}\left[\sum_{m=-\infty}^{i} 2^{\eta (0) m q(1+\epsilon)}\left\|\left(\sum_{j=1}^{\infty}\left|g_{j}\right|^{r}\right)^{\frac{1}{r}} {\bf{1}}_{m}\right\|_{p(\cdot)}^{q(1+\epsilon)}\right]\right\}^{\frac{1}{q(1+\epsilon)}} \\
& \lesssim \sup_{\epsilon>0}\sup _{L \leqslant 0, L \in \mathbb{Z}} 2^{-L \lambda}\left\{\epsilon^\theta\sum_{k=-\infty}^{L} \sum_{i=L+3}^{\infty} 2^{\frac{d q(1+\epsilon)(k-i)}{2}} \cdot 2^{i q (1+\epsilon)\lambda} E_{f}^{q(1+\epsilon)}\right\}^{\frac{1}{q(1+\epsilon)}} \\
& = \sup_{\epsilon>0} \sup _{L \leqslant 0, L \in \mathbb{Z}} 2^{-L \lambda}\left\{\epsilon^\theta\left(\sum_{k=-\infty}^{L} 2^{d q (1+\epsilon)k / 2}\right)\left(\sum_{i=L+3}^{\infty} 2^{(\lambda-d / 2) q(1+\epsilon) i}\right) E_{f}^{q(1+\epsilon)}\right\}^{\frac{1}{q(1+\epsilon)}} \\
& \lesssim \sup_{\epsilon>0} \sup _{L \leqslant 0, L \in \mathbb{Z}}\left\{2^{-L q(1+\epsilon) \lambda} 2^{d q(1+\epsilon) L / 2} 2^{(\lambda-d / 2) q(1+\epsilon) L} E_{f}^{q(1+\epsilon)}\right\}^{\frac{1}{q(1+\epsilon)}} \\
& =E_{f}.
\end{align*}

Hence we get  $E_{T}^{3} \lesssim E_{f}$.

For $0<q (1+\epsilon)\leqslant 1$, then using \eqref{eq3.4} in \eqref{eq3.7} we get

\begin{align*}
E_{T}^{3} & \lesssim \sup_{\epsilon>0} \sup _{L \leqslant 0, L \in \mathbb{Z}} 2^{-L \lambda}\left\{\sum_{k=-\infty}^{L} \sum_{i=k+2}^{\infty} 2^{\eta (0) i q(1+\epsilon)}\left\|\left(\sum_{j=1}^{\infty}\left|g_{j}\right|^{r}\right)^{\frac{1}{r}} {\bf{1}}_{i}\right\|_{p(\cdot)}^{q(1+\epsilon)} 2^{d q(k-i)}\right\}^{\frac{1}{q(1+\epsilon)}} \\
& \lesssim \sup_{\epsilon>0} \sup _{L \leqslant 0, L \in \mathbb{Z}} 2^{-L \lambda}\left\{\sum_{k=-\infty}^{L} \sum_{i=k+2}^{L+2} 2^{\eta (0) i q(1+\epsilon)}\left\|\left(\sum_{j=1}^{\infty}\left|g_{j}\right|^{r}\right)^{\frac{1}{r}} {\bf{1}}_{i}\right\|_{p(\cdot)}^{q(1+\epsilon)} 2^{d q(1+\epsilon)(k-i)}\right\}^{\frac{1}{q(1+\epsilon)}} \\
& +\sup_{\epsilon>0} \sup _{L \leqslant 0, L \in \mathbb{Z}} 2^{-L \lambda}\left\{\sum_{k=-\infty}^{L} \sum_{i=L+3}^{\infty} 2^{\eta (0) i q(1+\epsilon)}\left\|\left(\sum_{j=1}^{\infty}\left|g_{j}\right|^{r}\right)^{\frac{1}{r}} {\bf{1}}_{i}\right\|_{p(\cdot)}^{q(1+\epsilon)} 2^{d q(1+\epsilon)(k-i)}\right\} ^{\frac{1}{q(1+\epsilon)}}\\
= & :J_{3}+J_{4}.
\end{align*}

 Similarly we  conclude that $E_{T}^{3} \lesssim E_{f}$ holds for $0<q (1+\epsilon)\leqslant 1$.

Then we consider $H_{T}^{3}$. Similarly, we have

\begin{equation}\label{eq3.8}
H_{T}^{3} \lesssim 2^{-L \lambda}\left\{\epsilon^\theta \sum_{k=0}^{L}\left(\sum_{i=k+2}^{\infty} 2^{\eta _{\infty} i}\left\|\left(\sum_{j=1}^{\infty}\left|g_{j}\right|^{r}\right)^{\frac{1}{r}} {\bf{1}}_{i}\right\|_{p(\cdot)} 2^{d_{1}(k-i)}\right)^{q(1+\epsilon)}\right\}^{\frac{1}{q(1+\epsilon)}} 
\end{equation}

here $d_{1}=n \omega_{1}+\eta _{\infty}>0$.

If $1<q(1+\epsilon)<\infty$, then  Hölder's inequality yields

\begin{align*}
H_{T}^{3} \lesssim & 2^{-L \lambda}\left\{\epsilon^\theta\sum_{k=0}^{L}\left(\sum_{i=k+2}^{\infty} 2^{\eta _{\infty} q(1+\epsilon) i}\left\|\left(\sum_{j=1}^{\infty}\left|g_{j}\right|^{r}\right)^{\frac{1}{r}} {\bf{1}}_{i}\right\|_{p(\cdot)}^{q(1+\epsilon)} 2^{\frac{d_{1} q(1+\epsilon)(k-i)}{2}}\right)\right. \\
& \left.\times\left(\sum_{i=k+2}^{\infty} 2^{\frac{d_{1} (q(1+\epsilon))^{\prime}(k-i)}{2}}\right)^{\frac{q(1+\epsilon)}{(q(1+\epsilon))^{\prime}}}\right\}^{\frac{1}{q(1+\epsilon)}} \\
\lesssim & 2^{-L \lambda}\left\{\epsilon^\theta\sum_{k=0}^{\frac{1}{q}} \sum_{i=k+2}^{\infty} 2^{\eta _{\infty} q (1+\epsilon)i}\left\|\left(\sum_{j=1}^{\infty}\left|g_{j}\right|^{r}\right)^{\frac{1}{r}} {\bf{1}}_{i}\right\|_{p(\cdot)}^{q(1+\epsilon)} 2^{\frac{d_{1} q(1+\epsilon)(k-i)}{2}}\right\}^{\frac{1}{q(1+\epsilon)}} \\
\leqslant & 2^{-L \lambda}\left\{\epsilon^\theta\sum_{k=0}^{L} \sum_{i=k+2}^{L+2} 2^{\eta _{\infty} q (1+\epsilon)i}\left\|\left(\sum_{j=1}^{\infty}\left|g_{j}\right|^{r}\right)^{\frac{1}{r}} {\bf{1}}_{i}\right\|_{p(\cdot)}^{q(1+\epsilon)} 2^{\frac{d_{1} q(1+\epsilon)(k-i)}{2}}\right\}^{\frac{1}{q(1+\epsilon)}} \\
& +2^{-L \lambda}\left\{\epsilon^\theta\sum_{k=0}^{L} \sum_{i=L+3}^{\infty} 2^{\eta _{\infty} q (1+\epsilon)i}\left\|\left(\sum_{j=1}^{\infty}\left|g_{j}\right|^{r}\right)^{\frac{1}{r}} {\bf{1}}_{i}\right\|_{p(\cdot)}^{q(1+\epsilon)} 2^{\frac{d_{1} q(1+\epsilon)(k-i)}{2}}\right\} ^{\frac{1}{q(1+\epsilon)}}\\
= &: I_{5}+I_{6} .
\end{align*}

Because $d_{1}>0$, we have

\begin{align*}
I_{5} & =2^{-L \lambda}\left\{\epsilon^\theta \sum_{i=2}^{L+2} 2^{\eta _{\infty} q (1+\epsilon)i}\left\|\left(\sum_{j=1}^{\infty}\left|g_{j}\right|^{r}\right)^{\frac{1}{r}} {\bf{1}}_{i}\right\|_{p(\cdot)}^{q(1+\epsilon)} \sum_{k=0}^{i-2} 2^{\frac{d_{1} q(k-i)}{2}}\right\}^{\frac{1}{q(1+\epsilon)}} \\
& \lesssim 2^{-L \lambda}\left\{\epsilon^\theta\sum_{i=2}^{L+2} 2^{\eta _{\infty} q (1+\epsilon)i}\left\|\left(\sum_{j=1}^{\infty}\left|g_{j}\right|^{r}\right)^{\frac{1}{r}} {\bf{1}}_{i}\right\|_{p(\cdot)}^{q(1+\epsilon)}\right\}^{\frac{1}{q(1+\epsilon)}} \\
& \lesssim H_{f}.
\end{align*}

Since $d_{1}>0$ and $\lambda-d_{1} / 2<0$, we obtain

\begin{align*}
I_{6}&=2^{-L \lambda}\left\{\epsilon^\theta\sum_{k=0}^{L} \sum_{i=L+3}^{\infty} 2^{\eta _{\infty} q(1+\epsilon) i}\left\|\left(\sum_{j=1}^{\infty}\left|g_{j}\right|^{r}\right)^{\frac{1}{r}} {\bf{1}}_{i}\right\|_{p(\cdot)}^{q(1+\epsilon)} 2^{\frac{d_{1} q(1+\epsilon)(k-i)}{2}}\right\}^{\frac{1}{q(1+\epsilon)}}\\
& \leqslant 2^{-L \lambda}\left\{\epsilon^\theta\left(\sum_{k=0}^{L} 2^{\frac{d_{1} q(1+\epsilon) k}{2}}\right)\left(\sum_{i=L+3}^{\infty} 2^{\left(\lambda-d_{1} / 2\right) q(1+\epsilon) i}\right)\right. \\
& \left.\quad \times\left[\sum_{m=0}^{i} 2^{\eta _{\infty} q(1+\epsilon) m}\left\|2^{-i \lambda}\left(\sum_{j=1}^{\infty}\left|g_{j}\right|^{r}\right)^{\frac{1}{r}} {\bf{1}}_{m}\right\|_{p(\cdot)}^{q(1+\epsilon)}\right]\right\} ^{\frac{1}{q(1+\epsilon)}}\\
& \lesssim 2^{-L \lambda} 2^{d_{1} / 2 L} 2^{\left(\lambda-d_{1} / 2\right) L} H_{f} \\
& =H_{f} .
\end{align*}

For $0<{q(1+\epsilon)}\leqslant 1$, we use \eqref{eq3.4} in \eqref{eq3.8} and have

\begin{align*}
H_{T}^{3} & \lesssim 2^{-L \lambda}\left\{\epsilon^\theta\sum_{k=0}^{L} \sum_{i=k+2}^{\infty} 2^{\eta _{\infty} q (1+\epsilon)i}\left\|\left(\sum_{j=1}^{\infty}\left|g_{j}\right|^{r}\right)^{\frac{1}{r}} {\bf{1}}_{i}\right\|_{p(\cdot)}^{q(1+\epsilon)} 2^{d_{1} q(1+\epsilon)(k-i)}\right\}^{\frac{1}{q(1+\epsilon)}} \\
& \leqslant 2^{-L \lambda}\left\{\epsilon^\theta\sum_{k=0}^{L} \sum_{i=k+2}^{L+2} 2^{\eta _{\infty} q (1+\epsilon)i}\left\|\left(\sum_{j=1}^{\infty}\left|g_{j}\right|^{r}\right)^{\frac{1}{r}} {\bf{1}}_{i}\right\|_{p(\cdot)}^{q(1+\epsilon)} 2^{d_{1} q(1+\epsilon)(k-i)}\right\}^{\frac{1}{q(1+\epsilon)}} \\
& +2^{-L \lambda}\left\{\epsilon^\theta\sum_{k=0}^{L} \sum_{i=L+3}^{\infty} 2^{\eta _{\infty} q (1+\epsilon)i}\left\|\left(\sum_{j=1}^{\infty}\left|g_{j}\right|^{r}\right)^{\frac{1}{r}} {\bf{1}}_{i}\right\|_{p(\cdot)}^{q(1+\epsilon)} 2^{d_{1} q(1+\epsilon)(k-i)}\right\}^{\frac{1}{q(1+\epsilon)}} \\
= & :I_{J}+J_{6} .
\end{align*}

Similarly we can get $H_{T}^{3} \lesssim H_{f}$ where $0<q(1+\epsilon) \leqslant 1$.

Hence we completes our proof.
\end{proof}

Now we turn to prove Theorem $2.13$. Because the proofs of $B$-parts and $F$-parts are similar, we only prove $F$-parts below. Our proof will use the idea that comes from \cite{ref36}. To continue, we recall some lemmas.

\begin{lem} \cite{ref31}

 Let $\mu, \nu \in \mathcal{S}\left(\mathbb{R}^{n}\right),-1 \leqslant M \in \mathbb{Z}$, 
$$
D^{\tau} \widehat{\mu}(0)=0 \quad \text { for all } \quad|\tau| \leqslant M .
$$

Then for any $N>0$ there is a constant $C_{N}$ such that

$$
\sup _{z \in \mathbb{R}^{n}}\left|\mu_{t} * \nu(z)\right|(1+|z|)^{N} \leqslant C_{N} t^{M+1},
$$

where $\mu_{t}(x)=t^{-n} \mu\left(\frac{x}{t}\right)$ for all $0<t \leqslant 2$.
\end{lem}

\begin{lem}\cite{ref31}
 {\rm    If $\omega >0$ and  $q \in(0, \infty]$. Then for a sequence $\left\{g_{j}\right\}_{0}^{\infty}$, we have
$$
G_{j}=\sum_{\ell=0}^{\infty} 2^{-|\ell-j| \omega } g_{\ell}.
$$
Then
\begin{equation}\label{eq5}
\left\|\left\{G_{j}\right\}_{0}^{\infty}\right\|_{\ell_{q}} \leqslant C\left\|\left\{g_{j}\right\}_{0}^{\infty}\right\|_{\ell_{q}} .
\end{equation}

}
\end{lem}

\begin{lem}

If $\kappa , q \in(0, \infty], \omega >0$, $ s \in \mathbb{R}$, and $\eta , q,p $, are same as given in Theorem  $2.8$.  For a sequence $\left\{g_{j}\right\}_{0}^{\infty}$, we have

$$
G_{j}(x)=\sum_{k=0}^{\infty} 2^{-|k-j| \omega} g_{k}(x), \quad x \in \mathbb{R}^{n}. 
$$

Then there are some constants $C_{1}=C_{1}(q, \omega)$ and $C_{2}=C_{2}(p(\cdot), q, \omega)$ such that

\begin{equation}\label{eq3.10}
\left\|\left\{G_{j}\right\}_{j=0}^{\infty}\right\|_{{M\dot{K}_{\lambda, p(\cdot)}^{\eta (\cdot), q), \theta}}\left(\ell_{\kappa }\right)} \leqslant C_{1}\left\|\left\{g_{j}\right\}_{j=0}^{\infty}\right\|_{{M\dot{K}_{\lambda, p(\cdot)}^{\eta (\cdot), q), \theta}}\left(\ell_{\kappa }\right)}
\end{equation}

and

\begin{equation}\label{eq3.11}
\left\|\left\{G_{j}\right\}_{j=0}^{\infty}\right\|_{\ell_{\kappa }\left({M\dot{K}_{\lambda, p(\cdot)}^{\eta (\cdot), q), \theta}}\right)} \leqslant C_{2}\left\|\left\{g_{j}\right\}_{j=0}^{\infty}\right\|_{\ell_{\kappa }\left({M\dot{K}_{\lambda, p(\cdot)}^{\eta (\cdot), q), \theta}}\right)} .\tag{3.11}
\end{equation}
\end{lem}

\begin{proof}

 Firstly, \eqref{eq3.10} follows immediately from Lemma $3.3$. Next we prove \eqref{eq3.11} for $p(\cdot) \in \mathcal{P}^{0}\left(\mathbb{R}^{n}\right)$ and we separate it into two cases.

Case 1. $p_{-} \geqslant 1, q \geqslant 1$. Because $\|\cdot\|_{M\dot{K}_{\lambda, p(\cdot)}^{\eta (\cdot), q), \theta}}$ is a norm, we have

$$
\left\|G_{j}\right\|_{M\dot{K}_{\lambda, p(\cdot)}^{\eta (\cdot), q), \theta}} \leqslant \sum_{k=0}^{\infty} 2^{-|k-j| \omega}\left\|g_{k}\right\|_{M\dot{K}_{\lambda, p(\cdot)}^{\eta (\cdot), q), \theta}}.
$$

Using Lemma $3.4$, we get \eqref{eq3.11}.

Case $2$. If $q<1$, let $p_{0}<\min \left(p_{-}, q\right)$ then we get

\begin{align*}
\left\|G_{j}\right\|_{M\dot{K}_{\lambda, p(\cdot)}^{\eta (\cdot), q), \theta}}^{p_{0}} & =\left\|\left|G_{j}\right|^{p_{0}}\right\|_{M\dot{K}_{p_0\lambda, p(\cdot)/p_0}^{p_0\eta (\cdot),\left. q/p_0\right), p_0\theta}} \\
& \leqslant\left\|\sum_{k=0}^{\infty} 2^{-|k-j| p_{0} \omega}\left|g_{k}\right|^{p_{0}}\right\|_{M\dot{K}_{p_0\lambda, p(\cdot)/p_0}^{p_0\eta (\cdot),\left. q/p_0\right), p_0\theta}} \\
& \leqslant \sum_{k=0}^{\infty} 2^{-|k-j| p_{0} \omega}\left\|\left|g_{k}\right|^{p_{0}}\right\|_{M\dot{K}_{p_0\lambda, p(\cdot)/p_0}^{p_0\eta (\cdot),\left. q/p_0\right), p_0\theta}} .
\end{align*}

Consequently we get

\begin{align*}
 \left\|\left\{G_{j}\right\}\right\|_{\ell_{\kappa }\left({M\dot{K}_{\lambda, p(\cdot)}^{\eta (\cdot), q), \theta}}\right)}^{p_{0}}&=\left\|\left\{\left|G_{j}\right|^{p_{0}}\right\}\right\|_{\ell_{\kappa  / p_{0}}\left({M\dot{K}_{p_0\lambda, p(\cdot)/p_0}^{p_0\eta (\cdot),\left. q/p_0\right), p_0\theta}}\right)} \\
& \lesssim\left\|\left\{\left|g_{k}\right|^{p_{0}}\right\}\right\|_{\ell_{\kappa  / p_{0}}\left({M\dot{K}_{p_0\lambda, p(\cdot)/p_0}^{p_0\eta (\cdot),\left. q/p_0\right), p_0\theta}}\right)} \\
& =\left\|\left\{g_{k}\right\}\right\|_{\ell_{\kappa }\left({M\dot{K}_{\lambda, p(\cdot)}^{\eta (\cdot), q), \theta}}\right)}^{p_{0}}.
\end{align*}

By using the power $1 / p_{0}$, we get  \eqref{eq3.11}.
\end{proof}

\begin{lem}[\cite{ref9}, Theorem 6 ]
  {\rm  Let $\left\{\varphi_{j}\right\}_{j \in \mathbb{N}_{0}}$ is the resolution of unity,  $R \in \mathbb{N}$. Then there exists functions $\theta_{0}, \theta \in \mathcal{S}\left(\mathbb{R}^n\right)$ which satisfy  
$$
\begin{gathered}
\operatorname{supp} \theta, \operatorname{supp} \theta_{0} \subseteq\left\{y \in \mathbb{R}^n:|y| \leqslant 1\right\}, \\
\left|\widehat{\theta}_{0}(\varrho)\right|>0 \quad \text { on }\{|\varrho|<2 \varepsilon\}, \\
|\widehat{\theta}(\varrho)|>0 \quad \text { on }\left\{\frac{\varepsilon}{2}<|\varrho|<2 \varepsilon\right\}, \\
\int_{\mathbb{R}^n} y^{\gamma} \theta(x) \mathrm{d} y=0, \quad \forall \gamma, 0<|\gamma| \leqslant R,
\end{gathered}
$$
such that
$$
\widehat{\theta}_{0}(\varrho) \widehat{\varPsi}_{0}(\varrho)+\sum_{j=1}^{\infty} \widehat{\theta}\left(p^{-j} \varrho\right) \widehat{\varPsi}\left(p^{-j} \varrho\right)=1, \quad \forall \varrho \in \mathbb{R}^n,
$$
and  $\varPsi_{0}, \varPsi \in \mathcal{S}\left(\mathbb{R}^n\right)$ are given as 
$$
\widehat{\varPsi}_{0}(\varrho)=\frac{\varphi_{0}(\varrho)}{\widehat{\theta}_{0}(\varrho)}, \quad \widehat{\varPsi}(\varrho)=\frac{\varphi_{1}(2 \varrho)}{\widehat{\theta}(\varrho)}.
$$}
\end{lem}

\begin{proof}
Now we will give the proof of Theorem $2.13$.

{\bf{Step 1.}} Let $g\in  \mathcal{S}^{\prime}\left(\mathbb{R}^{n}\right)$, then 

$$
\|g\|_{M\dot{K}_{\lambda, p(\cdot)}^{\eta (\cdot), q), \theta} F_{\kappa }^{s}}^{(2)} \lesssim\|g\|_{M\dot{K}_{\lambda, p(\cdot)}^{\eta (\cdot), q), \theta}F_{\kappa }^{s}}^{(1)} \lesssim\|g\|_{M\dot{K}_{\lambda, p(\cdot)}^{\eta (\cdot), q), \theta}F_{\kappa }^{s}}^{(2)}.
$$

Using Lemmas $3.2$ and $3.5$, and the fact $r<\min \left\{p_-, \kappa \right\}$ for $N \in \mathbb{N}$. Then for $g\in  \mathcal{S}\left(\mathbb{R}^n\right)$,  then

\begin{align*}
\left(\int_{1}^{2}\left|2^{l s}\left(\Phi_{2^{-l} t}^{*} g\right)_{a}(x)\right|^{\kappa } \frac{\mathrm{d} t}{t}\right)^{r / \kappa } \lesssim & \sum_{k \in l+\mathbb{N}_{0}} 2^{(l-k)(N r-n+r s)} 2^{k r s} \\
& \times \mathcal{M}\left[\left(\int_{1}^{2}\left|\left(\left(\Phi_{k}\right)_{t} *g\right)(\cdot)\right|^{\kappa  \frac{\mathrm{d} t}{t}}\right)^{r / \kappa }\right](x).
\end{align*}

If $l \in \mathbb{N}$, $p_{0}=r \in\left(n / a,<\min \left\{p_{-}, \kappa \right\}\right), N>\max \{0,-s\}+a$ and $\omega:=$ $N+s-d / r>0$, then we get 

$$
\begin{aligned}
& \left(\int_{1}^{2}\left|2^{l s}\left(\Phi_{2^{-l}}^{*} g\right)_{a}(x)\right|^{\kappa } \frac{\mathrm{d} t}{t}\right)^{r / \kappa } \\
& \quad \lesssim \sum_{k \in l+\mathbb{N}_{0}} 2^{-\omega r|l-k|} 2^{k r s} \mathcal{M}\left[\left(\int_{1}^{2}\left|\left(\left(\Phi_{k}\right)_{t} *g\right)(\cdot)\right|^{\kappa } \frac{\mathrm{d} t}{t}\right)^{r / \kappa }\right](x) .
\end{aligned}
$$

Using Lemma $3.5$ in ${M\dot{K}_{r\lambda, p(\cdot)/r}^{r\eta (\cdot),\left. q/r\right), r\theta}}\left(\ell_{\kappa  / r}\right)$, we obtain

\begin{align*}
& \left\|\left\{\left(\int_{1}^{2}\left|2^{l s}\left(\Phi_{2-l_{t}}^{*} g\right)_{a}(x)\right|^{\kappa } \frac{\mathrm{d} t}{t}\right)^{r / \kappa }\right\}_{l \in \mathbb{N}}\right\|_{M\dot{K}_{r\lambda, p(\cdot)/r}^{r\eta (\cdot),\left. q/r\right), r\theta}\left(\ell_{\kappa  / r}\right)} \\
& \quad \lesssim\left\|\left\{\mathcal{M}\left[\left(\int_{1}^{2}\left|2^{k s}\left(\left(\Phi_{l}\right)_{t} *g\right)(\cdot)\right|^{\kappa } \frac{\mathrm{d} t}{t}\right)^{r / \kappa }\right]\right\}_{l \in \mathbb{N}}\right\|_{M\dot{K}_{r\lambda, p(\cdot)/r}^{r\eta (\cdot),\left. q/r\right), r\theta}\left(\ell_{\kappa  / r}\right)}
\end{align*}

Theorem $2.8$ yields

\begin{align*}
& \left\|\left\{\left(\int_{1}^{2}\left|2^{l s}\left(\Phi_{2^{-l}}^{*} g\right)_{a}(x)\right|^{\kappa } \frac{\mathrm{d} t}{t}\right)^{r / \kappa }\right\}_{l \in \mathbb{N}}\right\|_{M\dot{K}_{r\lambda, p(\cdot)/r}^{r\eta (\cdot),\left. q/r\right), r\theta}\left(\ell_{\kappa  / r}\right)}\\
& \lesssim\left\|\left\{\left(\int_{1}^{2}\left|2^{k s}\left(\left(\Phi_{l}\right)_{t} *g\right)(\cdot)\right|^{\kappa } \frac{\mathrm{d} t}{t}\right)^{r / \kappa }\right\}_{l \in \mathbb{N}}\right\|_{M\dot{K}_{r\lambda, p(\cdot)/r}^{r\eta (\cdot),\left. q/r\right), r\theta}\left(\ell_{\kappa  / r}\right)} \\
& =\left\|\left\{\left(\int_{1}^{2}\left|2^{k s}\left(\left(\Phi_{l}\right)_{t} *g\right)(\cdot)\right|^{\kappa } \frac{\mathrm{d} t}{t}\right)^{1 / \kappa }\right\}_{l \in \mathbb{N}}\right\|_{M\dot{K}_{\lambda, p(\cdot)}^{\eta (\cdot), q), \theta}\left(\ell_{\kappa }\right)}^{r}.
\end{align*}

Hence, we have

\begin{align*}
 &\left\| \left(\int_{0}^{1}\left|\lambda^{-s}\left(\Phi_{\lambda}^{*} g\right)_{a}(\cdot)\right|^{\kappa } \frac{d \lambda}{\lambda}\right)^{1 / \kappa } \right\|_{M\dot{K}_{\lambda, p(\cdot)}^{\eta (\cdot), q), \theta}} \\
& \approx\left\|\left(\sum_{l=1}^{\infty} \int_{1}^{2}\left|2^{l s}\left(\Phi_{2-l_{t}}^{*} g\right)_{a}(\cdot)\right|^{\kappa } \frac{\mathrm{d} t}{t}\right)^{1 / \kappa }\right\|_{M\dot{K}_{\lambda, p(\cdot)}^{\eta (\cdot), q), \theta}} \\
& \lesssim\left\|\left\{\left(\int_{1}^{2}\left|2^{l s} \Phi_{2^{-l} {t}} *g(\cdot)\right|^{\kappa } \frac{\mathrm{d} t}{t}\right)^{1 / \kappa }\right\}_{l \in \mathbb{N}}\right\|_{M\dot{K}_{\lambda, p(\cdot)}^{\eta (\cdot), q), \theta}\left(\ell_{\kappa }\right)} \\
& \approx\left\|\left(\int_{0}^{1}\left|\lambda^{-s} \Phi_{\lambda} *g(\cdot)\right|^{\kappa } \frac{d \lambda}{\lambda}\right)^{1 / \kappa }\right\|_{M\dot{K}_{\lambda, p(\cdot)}^{\eta (\cdot), q), \theta}}.
\end{align*}

This proves $\|g\|_{M\dot{K}_{\lambda, p(\cdot)}^{\eta (\cdot), q), \theta} F_{\kappa }^{s}}^{(2)} \lesssim\|g\|_{M\dot{K}_{\lambda, p(\cdot)}^{\eta (\cdot), q), \theta} F_{\kappa }^{s}}^{(1)}$.

{\bf{Step 2.}} Suppose that $\Psi_{0}, \Psi \in \mathcal{S}^{\prime}\left(\mathbb{R}^{n}\right)$ and  $g\in  \mathcal{S}^{\prime}\left(\mathbb{R}^{n}\right)$.

\begin{equation}\label{eq3.12}
\|g\|_{M\dot{K}_{\lambda, p(\cdot)}^{\eta (\cdot), q), \theta} F_{\kappa }^{s}\left(\mathbb{R}^{n}, \Psi\right)}^{(4)} \lesssim\|g\|_{M\dot{K}_{\lambda, p(\cdot)}^{\eta (\cdot), q), \theta} F_{\kappa }^{s}\left(\mathbb{R}^{n}, \Phi\right)}^{(2)}. 
\end{equation}

By applying  Lemmas $3.5$ and $3.2$, and the fact  $\omega =\min \{1, S+1-s\}$,then for  $g \in \mathcal{S}$, we get

\begin{equation}\label{eq3.13}
2^{l s}\left(\Psi_{l}^{*} g\right)_{a}(x) \leqslant C \sum_{k \in \mathbb{N}_{0}} 2^{-|k-l| \omega} 2^{k s}\left(\Phi_{2-{ }^{k} t}^{*} g\right)_{a}(x), x \in \mathbb{R}^{n} \text { and } t \in[1,2] .
\end{equation}

If $\kappa  \geqslant 1$. By using the  $\left(\int_{1}^{2}|\cdot|^{\kappa } \mathrm{d} t / t\right)^{1 / \kappa }$, we get 

$$
2^{l s}\left(\Psi_{l}^{*} g\right)_{a}(x) \lesssim \sum_{k \in \mathbb{N}_{0}} 2^{-|k-l| \omega} 2^{k s}\left(\int_{1}^{2}\left|\left(\Phi_{2^{-k} t}^{*} g\right)_{a}(x)\right|^{\kappa } \frac{\mathrm{d} t}{t}\right)^{1 / \kappa }.
$$

Applying Lemma $3.5$, we obtain

$$
\left\|\left\{2^{l s}\left(\Psi_{l}^{*} g\right)_{a}\right\}_{l \in \mathbb{N}}\right\|_{M\dot{K}_{\lambda, p(\cdot)}^{\eta (\cdot), q), \theta}\left(\ell_{\kappa }\right)} \lesssim\left\|\left(\sum_{k=1}^{\infty} 2^{k s \kappa } \int_{1}^{2}\left|\left(\Phi_{2^{-k} t}^{*} g\right)_{a}(x)\right|^{\kappa } \frac{\mathrm{d} t}{t}\right)^{1 / \kappa }\right\|_{M\dot{K}_{\lambda, p(\cdot)}^{\eta (\cdot), q), \theta}}.
$$

Hence we obtain the required result.

If $\kappa <1$, then  $\left(\int_{1}^{2}|\cdot|^{\kappa } \mathrm{d} t / t\right)^{1 / \kappa }$ is not the norm. Thus we get 

$$
\left(2^{l s}\left(\Psi_{l}^{*} g\right)_{a}(x)\right)^{\kappa } \lesssim \sum_{k \in \mathbb{N}_{0}} 2^{-\kappa |k-l| \omega} 2^{k s \kappa } \int_{1}^{2}\left|\left(\Phi_{2^{-k} t}^{*} g\right)_{a}(x)\right|^{\kappa } \frac{\mathrm{d} t}{t}.
$$

Convolution $(\gamma * \eta )_{\ell}$ of the sequences yields

$$
\gamma_{k}=2^{-|k| \omega \kappa } \quad \text { and } \quad \tau_{k}=2^{k s \kappa } \int_{1}^{2}\left|\left(\Phi_{2-k t}^{*} g\right)_{a}(x)\right|^{\kappa } \frac{\mathrm{d} t}{t}
$$

For $x \in \mathbb{R}^{n}$, using $\ell_{1}$-norm  gives as
$$
\begin{aligned}
\left\|2^{l s}\left(\Psi_{l}^{*} g\right)_{a}(x)\right\|_{\ell_{\kappa }}^{\kappa } & \leqslant\|\gamma\|_{\ell_{1}} \cdot\|\tau\|_{\ell_{1}} \\
& \lesssim \sum_{k=1}^{\infty} 2^{k s \kappa } \int_{1}^{2}\left|\left(\Phi_{2^{-k} t}^{*} g\right)_{a}(x)\right|^{\kappa } \frac{\mathrm{d} t}{t} .
\end{aligned}
$$

By taking  $(\cdots)^{1 / \kappa }$ and using ${\dot{K} ^{\eta (\cdot), q),\theta}_{p(\cdot)}}$-norm. We obtain desired result \eqref{eq3.12}.

Similarly, for any $g\in  \mathcal{S}^{\prime}\left(\mathbb{R}^{n}\right)$, we obtain

$$
\|g\|_{M\dot{K}_{\lambda, p(\cdot)}^{\eta (\cdot), q), \theta} F_{\kappa }^{s}\left(\mathbb{R}^{n}, \Phi\right)}^{(2)} \lesssim\|g\|_{M\dot{K}_{\lambda, p(\cdot)}^{\eta (\cdot), q), \theta} F_{\kappa }^{s}\left(\mathbb{R}^{n}, \Psi\right)}^{(4)}.
$$

{\bf{Step 3.}} Using $t=1$ in {\bf{Step 1}}, we get 

$$
\|g\|_{M\dot{K}_{\lambda, p(\cdot)}^{\eta (\cdot), q), \theta} F_{\kappa }^{s}}^{(5)} \lesssim\|g\|_{M\dot{K}_{\lambda, p(\cdot)}^{\eta (\cdot), q), \theta}F_{\kappa }^{s}}^{(4)} \lesssim\|g\|_{M\dot{K}_{\lambda, p(\cdot)}^{\eta (\cdot), q), \theta}F_{\kappa }^{s}}^{(5)}.
$$

{\bf{Step 4.}} We show \eqref{eq2.10} is equivalent to the rest.

First, we will show that for any $g\in  \mathcal{S}^{\prime}\left(\mathbb{R}^{n}\right)$

\begin{equation}\label{eq3.14}
\|g\|_{M\dot{K}_{\lambda, p(\cdot)}^{\eta (\cdot), q), \theta} F_{\kappa }^{s}\left(\mathbb{R}^{n}\right)}^{(2)} \lesssim\|g\|_{M\dot{K}_{\lambda, p(\cdot)}^{\eta (\cdot), q), \theta} F_{\kappa }^{s}}^{(3)} .
\end{equation}

For $0<r<\min \left\{p_{-}, \kappa \right\}$, see \cite{ref36}, there exists a positive constant $C$ such that for any $g\in  \mathcal{S}^{\prime}\left(\mathbb{R}^{n}\right)$,

\begin{align*}
& \left(\int_{1}^{2}\left|\left(\Psi_{2^{-l}}^{*} g\right)_{a}(x)\right|^{\kappa } \frac{\mathrm{d} t}{t}\right)^{r / \kappa } \\
& \quad \leqslant C \sum_{k \in \mathbb{N}_{0}} 2^{-k N s} 2^{(k+l) n} \int_{\mathbb{R}^{n}} \frac{\left(\int_{1}^{2} \int_{|z|<2^{-(k+l)} t}\left|\left(\left(\Phi_{k+l}\right)_{t} *g\right)(z+y)\right|^{\kappa } \mathrm{d} z \frac{\mathrm{d} t}{t^{n+1}}\right)^{r / \kappa }}{\left(1+2^{l}|x-y|\right)^{a r}} \mathrm{~d} y.
\end{align*}

Let $a r>n$, we get 

$$
g_{l}(y):=\frac{2^{n l}}{\left(1+2^{l}|y|\right)^{a r}}, \forall y \in \mathbb{R}^{n}.
$$

Hence we get

\begin{align*}
& \left(\int_{1}^{2}\left|2^{l s}\left(\Phi_{2^{-l} l_{t}}^{*} g\right)_{a}(x)\right|^{\kappa } \frac{\mathrm{d} t}{t}\right)^{r / \kappa } \\
& \lesssim \sum_{k \in \mathbb{N}_{0}} 2^{-k N r} 2^{k n} 2^{l s r}\left[g_{l} *\left(\int_{1}^{2} \int_{|z|<2^{-(k+l) t}}\left|\left(\left(\Phi_{k+l}\right)_{t} *g\right)(z+\cdot)\right|^{\kappa } \mathrm{d} z \frac{\mathrm{d} t}{t^{n+1}}\right)^{r / \kappa }\right](x).
\end{align*}

By applying the majorant property  see \cite{ref34} to obtain

\begin{align*}
& \left(\int_{1}^{2}\left|2^{l s}\left(\Phi_{2^{-l} t}^{*} g\right)_{a}(x)\right|^{\kappa } \frac{\mathrm{d} t}{t}\right)^{r / \kappa } \\
& \lesssim \sum_{k \in \mathbb{N}_{0}} 2^{l s r} 2^{k(-N r+n)} \mathcal{M}\left[\left(\int_{1}^{2} \int_{|z|<2^{-(k+l) t}}\left|\left(\left(\Phi_{k+l}\right)_{t} *g\right)(z+\cdot)\right|^{\kappa } \mathrm{d} z \frac{\mathrm{d} t}{t^{n+1}}\right)^{r / \kappa }\right](x).
\end{align*}

An index shift on the right-hand side gives

\begin{align*}
& \left(\int_{1}^{2}\left|2^{l s}\left(\Phi_{2^{-l}}^{*} g\right)_{a}(x)\right|^{\kappa } \frac{\mathrm{d} t}{t}\right)^{r / \kappa } \\
& \lesssim \sum_{k \in l+\mathbb{N}_{0}} 2^{l s r} 2^{(k-l)(-N r+n)} \mathcal{M}\left[\left(\int_{1}^{2} \int_{|z|<2^{-k} t}\left|\left(\left(\Phi_{k}\right)_{t} *g\right)(z+\cdot)\right|^{\kappa } \mathrm{d} z \frac{\mathrm{d} t}{t^{n+1}}\right)^{r / \kappa }\right](x) \\
& =\sum_{k \in l+\mathbb{N}_{0}} 2^{(l-k)(N r-n+r s)} 2^{k r s} \mathcal{M}\left[\left(\int_{1}^{2} \int_{|z|<2^{-k t}}\left|\left(\left(\Phi_{k}\right)_{t} *g\right)(z+\cdot)\right|^{\kappa } \mathrm{d} z \frac{\mathrm{d} t}{t^{n+1}}\right)^{r / \kappa }\right](x).
\end{align*}

It is simple to note that $\|g\|_{M\dot{K}_{\lambda, p(\cdot)}^{\eta (\cdot), q), \theta} F_{\kappa }^{s}}^{(3)} \lesssim\|g\|_{M\dot{K}_{\lambda, p(\cdot)}^{\eta (\cdot), q), \theta}  F_{\kappa }^{s}}^{(2)}$, since for any $t>0$

$$
\frac{1}{t^{n}} \int_{|z|<t}\left|\left(\Phi_{t} *g\right)(x+z)\right| \mathrm{d} z \lesssim \sup _{|z|<t} \frac{\left|\left(\Phi_{t} *g\right)(x+z)\right|}{(1+1 / t|z|)^{a}} \lesssim\left(\Phi_{t}^{*} g\right)_{a}(x).
$$

Hence we  completes the proof.
\end{proof}

\section{Ethics declarations}

	\section*{Conflict of interest}
	
	The authors declare that they have no known competing financial interests or personal relationships that could have appeared to influence the work reported in this paper.

 \section*{Ethics approval and consent to participate}
This manuscript has not and will not be submitted to more than one journal for simultaneous consideration. The submitted work is original and will not be published elsewhere.

	\section*{Funding }
	Authors state no funding involved.
	
	\section*{Availability of data  }
	No data is available for this study.

\bibliographystyle{elsarticle-num}

\end{document}